\documentclass[draft,12pt,leqno,twoside]{amsart}
 \usepackage{amssymb,amscd}
\usepackage{epsf, amsmath,amsbsy,amsfonts, amsthm, amssymb,upref,yfonts}
 \usepackage{amsrefs}
\newtheorem{thm}{Theorem}[section]
\newtheorem{lem}[thm]{Lemma}
\newtheorem{cor}[thm]{Corollary}
\newtheorem{prop}[thm]{Proposition}
\theoremstyle{definition}
\newtheorem{defin}[thm]{Definition}
\newtheorem{ex}[thm]{Example}
\newtheorem{exc}[thm]{Exercise}

\newtheorem*{defn}{Definition}          
\newtheorem*{prob}{Problem}    

\newtheorem*{rem}{Remark}
\newtheorem{remark}[thm]{Remark}

\newcommand{\N}{\mathbb{N}}

\newcommand{\Z}{\mathbb{Z}}
\renewcommand{\P}{\mathbb{P}}
\newcommand{\R}{\mathbb{R}}

\newcommand{\E}{\mathbb{E}}

\newcommand{\G}{\mathbb{G}}

\newcommand{\cA}{\mathcal A}

\newcommand{\cF}{\mathcal F}

\newcommand{\cL}{\mathcal L}

\newcommand{\cM}{\mathcal M}

\newcommand{\cP}{\mathcal P}

\newcommand{\cT}{\mathcal T}

\newcommand{\supp}{\text{\rm supp}}
\newcommand{\spa}{\text{\rm span}}
\newcommand{\diam}{\text{\rm diam}}
\newcommand{\dist}{\text{\rm dist}}

\newcommand{\length}{\text{\rm length}}
\newcommand{\Lip}{\text{\rm Lip}}
\newcommand{\Leaf}{\text{\rm Leaf }}

\newcommand{\sign}{\text{\rm sign}}

\newcommand{\Wa}{\text{\rm Wa}}

 \newcommand{\kin}{\!\in\!}
\newcommand{\kplus}{\!+\!}

\newcommand{\keq}{\!=\!}
\newcommand{\kle}{\!<\!}

\newcommand{\ksubset}{\!\subset\!}

\newcommand{\vp}{\varepsilon}

\newcommand{\ie}{{\it i.e.,}} 

\newcommand{\DS}{\text{\rm DS}}

\newcommand{\tc}{\text{\rm tc}}

\newcommand{\Per}{\text{\rm Per}}

\newcommand{\VF}{\text{\rm VF}}
\newcommand{\CVF}{\text{\rm CVF}}

\newcommand{\la}{\langle}
\newcommand{\ra}{\rangle}

\begin{document}
\thanks{The author was supported by the National Science Foundation under Grant Number DMS-2054443.}
]

\title{Transportation cost spaces and their embeddings into $L_1$, a Survey}

\author{Th.~Schlumprecht}
\address{Department of Mathematics, Texas A\&M University, College
  Station, TX 77843, USA and Faculty of Electrical Engineering, Czech
  Technical University in Prague, Zikova 4, 166 27, Prague, Czech
  Republic}
\begin{abstract}These notes present a basic survey on Transportation cost spaces (aka Lipschitzfree spaces, Wasserstein spaces) and their 
bi-Lipschitz and linear embeddings into $L_1$ spaces. To make these notes as self-contained as possible, we added the proofs of several relevant results from computational graph theory in the appendix.
\end{abstract}
\maketitle

\tableofcontents
\allowdisplaybreaks
\section{Introduction}
Transportation cost spaces are of high theoretical interest,  and they also are fundamental in applications in many areas of applied mathematics, engineering, physics, computer science, finance, social sciences, and more. For this reason, they appear under several different other  names in the literature:
 {\em Lipschitz  Free Spaces},
 {\em Wasserstein Spaces}  or more precisely  {\em Wasserstein 1-Spaces},
 {\em Arens Eals spaces},
 {\em Earthmover spaces}.
 
Depending on the context, we will use the names {\em Transportation cost space} and {\em Lipschitz free space}. 
 If $\sigma$ and $\tau $ are two probabilities on a metric space $(M,d)$, both having finite support, then the difference $\sigma-\tau$ can be seen as an element of the Lipschitz free space  $\cF(M)$
  of $(M,d)$. Thus $d_{\Wa}(\sigma,\tau) = \|\sigma-\tau\|$ is a metric on the space of probabilities on $M$ with finite support. We will call that metric space 
  {\em Wasserstein space of $M$}.
  
 We investigate the relationship between the  Transportation cost space $\cF(M)$ over a metric space $(M,d)$ and $L_1$ spaces; in particular, we concentrate on the question of how 
  well $\cF(M)$ is embeddable into an $L_1$. We will mostly consider finite metric spaces. 
  
  For better readability, we tried to make these notes as self-contained as possible and added the proof of some results from computational graph theory in the appendix.

\section{Basic facts about Lipschitz free spaces}
In this section, we will introduce the {\em Lipschitz free space } $\cF(M)$ over a metric space $(M,d)$ and only concentrate on the very basic facts, which we will need later.
 For a more comprehensive account of the properties of $\cF(M)$, we refer the reader to the papers \cite{Kalton2004,GodefroyKalton2003}.
  For two metric spaces $(M,d_M)$ and $(N,d_N)$, and a map $\phi$ we denote the  Lipschitz constant by $\Lip(\phi)$, \ie
  $$\Lip(\phi)=\sup_{m,m'\in M, m\not=m'} \frac{  d_N(\phi(m),\phi(m')}{d_M(m,m')}.$$

 Let $(M,d)$ be a metric space with a special point $0$. Let $\Lip_0(M)$ be the Banach space of Lipschitz functions $f$ on $M$ 
   with value in $\R$, for which $f(0)=0$. In that case $\Lip(\cdot)$ becomes a norm on $\Lip_0(M)$
   and we write 
   for $f\in \Lip_0(M)$
   $$\|f\|_{\Lip}= \Lip(f)= \sup_{x\not =y, x,y\in M} \frac{|f(y)-f(x)|}{d(x,y)}.$$
      
\begin{prop}\label{P:6.1.2} Assume that $M$ is a metric space and $X$  a Banach space. Then 
$\|\cdot\|_L$ is a norm on $\Lip_0(M,X)$ which turns $\Lip_0(M,X)$ into a Banach space.
\end{prop}
      
    Let $\cM(M)$
   be the space of linear combinations of Dirac measures  $\mu$ on $M$.
 The elements of  $\cM(M)$  can be seen as elements of the dual of $\Lip_0(M)$ 
 and we define 
 $\cF(M)$ to be closure of $\cM(M)$ in $\Lip^*_0(M)$. Thus $\cF(M)$ is the completion of $\cM_0(M)$ with respect to  the norm 
  $$\|\mu\|_\cF= \sup_{\substack{ f\in \Lip_0(M)\\ \|f\|_\Lip\le 1}} \int f(x) \,d\mu(x)= \sup_{\substack{ f\in \Lip_0(M)\\ \|f\|_\Lip\le 1} } \sum_{j=1}^n a_j f(x_j)\text{ for $\mu=\sum_{j=1}^n a_j \delta_{x_j}\in \cM(M)$}.$$
  
  $\cF(M)$ can be seen as a  ``linearization of $M$''  as the following observation suggests.
  \begin{prop}\label{P:6.1.4}  Let $M$ be a metric space. The map $\delta_M: M\to \cF(M)$, $m\mapsto \delta_m$,  is an isometry. 

We will, from now on, identify elements of $M$ with their image in $\cF(M)$ under $\delta_M$.
\end{prop}
\begin{proof} For $m,m'\in M$ 
$$\|\delta_m-\delta_{m'}\|_{\Lip^*}=\sup_{f\in \Lip_0(M),  \|f\|_L\le 1} \|f(m)-f(m')\|\le d(m,m').$$
On the other hand, define for $m\in M$
$$f_m(m')= d(m,m')- d(m,0), \text{ for $m'\in M$}.$$
then $f \in \Lip_0(M)$ with $\|f\|_L=1$, and 
$$\la  \delta_{m'}-\delta_{m}, f\ra = f(m')-f(m)=d(m',m),$$
and thus $\|\delta_m-\delta_{m'} \|_{L^*}\ge d(m,m')$.
\end{proof}
 \begin{remark} The  measure $\delta_0$, as elment of $\Lip_0^*(M)$ is the $0$-functional. Therefore the family  $(\delta_m:m\in M)$ is not linear independent in $\Lip_0^*(M)$.
 But it is easy to see that the family  $(\delta_m:m\in M\setminus\{0\})$ is linear independent, and the linear span of it is dense in $\cF(M)$.
 
 Let $\cM_0(M)$ be the  elements $\mu\in \cM(M)$ for which $\mu(M)=0$.
 After adding an appropriate  multiple of $\delta_0$  to $\mu\in \cM$,  which will not change how $\mu$ acts on elements of $\Lip_0(M)$, we can assume that $\mu\in \cM_0(M)$ and thus
 we also can see $\cF(M)$  as the closure of $\cM_0(M)$ in $\Lip^*_0(M)$.
 \end{remark}

 We will show    (cf.  \cite{Weaver1999}*{Section 2}) that  $\Lip_0(M)$ is, in a natural way, isometrically equivalent to the dual of $\cF(M)$.
 The following observation is easy to see and crucial.
 
%
%
%
%
%

\begin{prop}\label{P:6.1.1} Assume that 
$(f_i)_{i\in I}$ is a net in $ \Lip_0(M)$, with $\|f\|_{\Lip}\le 1$, for $i\in I$ and   $f(m)=\lim_{i\in I} f_i(m)$ exists for every $m\in M$, then 
$f\in  \Lip_0(M)$, and $\|f\|_{\Lip}\le 1$.
\end{prop}

\begin{thm}\label{T:6.2.4} 
Let $M$ be a metric space and $Z$ a Banach space then the canonical map 
$$\Lip_0(M)\to \cF^*(M), \quad f\mapsto \chi_f, \text{ with }\chi_f(\mu)=\la \mu, f\ra ,\text{ for $\mu\in \cF(M)$}$$ 
is an isometry  from $\Lip_0(M)$ onto $\cF^*(M)$.

\end{thm} 

We will need the following well known result by Dixmier:
   
\begin{thm}{\rm \cite{Dixmier1948,Ng1971} }\label{T:6.2.5}
Assume that $U$ is a Banach space and $V$ is a closed subspace of $U^*$ so that
   $B_U$, the unit ball in $U$, is compact in the topology $\sigma(U,V)$, the topology on $U$ generated by $V$.

Then $V^*$ is isometrically isomorphic to $U$ and the map
$$T: U\to V^*,\quad u\mapsto F_u, \text{ with $F_u(v)=\la u,v\ra $, for $u\in U$ and $v\in V$,}$$
is an isometrical isomorphism onto $V^*$.
\end{thm} 

\begin{proof}[Proof of Theorem \ref{T:6.2.4}] We will verify that the assumptions of Theorem \ref{T:6.2.5} hold for $U=\Lip_0(M)$ and $V=\cF(M)\subset \Lip^*_0(M)$.

Let $(f_i)_{i\in I}\subset B_{ \Lip_0(M)}$  be a net in $\subset B_{ \Lip_0(M)}$. We have to show that there is a subnet $(g_j)_{j\in J}$ of $(f_i)_{i\in I }$ 
 which converges to some element  $f\in B_{\Lip_0(M)}$ with respect to $\sigma\big(\Lip_0(M), \cF(M)\big)$. Let
  $$K=\prod_{m\in M} [0, d(0,m)]$$
 which is compact in the product topology.  It follows  that  some  subnet $(g_j)_{j\in J}$ of $(f_i)_{i\in I }$  is pointwise converging to $g\in K$, and  from
 Proposition  \ref{P:6.1.1}  it follows that  $g\in \Lip_0(M)$, and $\|g\|_{\Lip_0(M)}\le 1$.
  Thus $\la\mu, g_j\ra$ converges to $\la \mu, g\ra$ for each $\mu$ in the linear span of $(\delta_m:m\in M)$, and since  $(f_i)_{i\in I}$ is bounded it follows that 
   $\la\mu, g_j\ra$ converges to $\la \mu, g\ra$ for all $\mu\in \cF(M)$. We deduce therefore our claim from Theorem \ref{T:6.2.5}.
 \end{proof}

\subsection{Example}
\begin{ex}\label{Ex:1.6} Let $M=\R$. We want to find a concrete representation of  the space $\cF(\R)$. 
 Every Lipschitz function on $\R$  is absolutely continuous and   thus, by the Fundamental Theorem of Calculus \cite[Theorem 3.35]{Folland1999}
 almost everywhere differentiable. Moreover the derivative 
 $$D: \Lip_0(\R)\to L_\infty(\R), \quad  f\mapsto f'$$
 is an isometry whose inverse is 
 $$I: L_\infty(\R)\to \Lip_0(\R),\quad  f\mapsto F,\text{ with }F(x)=\int_0^x  f(t)\, dt\text{ for $x\in \R$.}$$
 We claim that the map
  $\delta_x\mapsto 1_{[0,x]}$, for $x\ge 0$ and  $\delta_x\mapsto 1_{[x,0]}$, if $x<0$, extends to an isometric isomorphism between  $\cF(\R)$ and $L_1(\R)$.
  
  Indeed, we represent the span of the $\delta_x$, $x\in \R$ by 
  $$\tilde\cF=\left\{ \sum_{j=-m}^{n-1} \xi_j (\delta_{x_{j+1}} -\delta_{x_{j}}): \begin{matrix} m,n\kin\N,\, x_{-m}\kle x_{-m+1}\kle\ldots\kle x_{-1}\kle x_0\keq0\kle x_1\kle\ldots x_n\\  \text{ and $(\xi_j)_{j=-m}^n\subset \R$}\end{matrix} \right\}.$$ 
  Then, we consider the map 
  $$T:\tilde \cF \to L_1(\R),\quad \sum_{j=-m}^{n-1} \xi_j (\delta_{x_{j+1}} -\delta_{x_j})\mapsto\sum_{j=-m}^{n-1} \xi_j 1_{[x_j, x_{j+1})}.$$
  Since $T$ has a dense image, we only need to show that $T$ is an isometry.

 For $\mu=    \sum_{j=-m}^{n-1} \xi_j (\delta_{x_{j+1}}-\delta_{x_j} )\in\tilde\cF$ 
  it follows that
  \begin{align*}
 \big\|T(\mu)\|&= \sup_{f\in L_\infty(\R), \|f\|_\infty\le 1}  \Bigg| \sum_{j=-m}^{n-1} \xi_j \int_{x_{j}}^{x_{j+1}} f(t)\, dt \Bigg|\\
 &= \sup_{f\in L_\infty(\R), \|f\|_\infty\le 1}   \sum_{j=-m}^{n-1}\xi_j\big( I(f)(x_{j+1})- I(f)(x_j)\big)\\
 &=\sup_{F\in \Lip_0(\R), \|f\|_\infty\le 1}   \sum_{j=-m}^{n-1}\xi_j\big( F(x_{j+1})- F(x_j)\big)= \|\mu\|_{\Lip^*},
 \end{align*}
  which verifies our claim.
\end{ex}
 \begin{ex}\label{Ex:6.2.7} Similarly the following can be shown. If $M=\Z$ with its usual metric, then
 $$T: \Lip_0(\Z)\to \ell_\infty(\Z),   \,\, f\mapsto x_f, \text{ with } x_f(n)=f(n)-f(n-1), \,\, \text{ for } n\kin \Z,$$
 is an isometric isomorphism onto $\ell_\infty(\Z)$, and it is the adjoint 
 of the operator 
 $$S:\ell_1(\Z) \to \cF(\Z), \quad e_n \mapsto \delta_{n}-\delta_{n-1},$$
 which is also an isometric isomorphism.  
 \end{ex}
 \subsection{The Universality property of $\cF(M)$}
 An important property is the following extension property of $\cF(M)$. 
 \begin{prop}\label{P:6.2.9} Let $M$ and $N$ be metric spaces, with special $0$-points $0_M$ and $0_N$. Then every  Lipschitz map $\varphi:M\to N$, with $\varphi(0_M)=0_N$, can be extended
 to a linear  bounded map $\hat \varphi:  \cF(M)\to \cF(N)$ (meaning that $\hat\varphi(\delta_m)=\delta_{\varphi(m)}$, for $m\in M$) so that for the operator norm of $\hat\varphi$ we have
 $$\|\hat\varphi\|_{\cF(M)\to \cF(N)} \le \Lip(\varphi).$$
 \end{prop}
 \begin{proof} Let $\lambda =\Lip(\varphi)$. We define $$\varphi^{\#}: \Lip_0(N)\to \Lip_0(M), \qquad f\mapsto f\circ \varphi.$$
 Thus, $\varphi^\#$ is a linear bounded operator, with $\|\varphi^\#\|\le \lambda$. We claim that $\varphi^\#$   is also $w^*$ continuous,
 more precisely $\varphi^\#$ is  $\sigma(\Lip_0(N),\cF(N))$- $\sigma(\Lip_0(M),\cF(M))$-continuous.
 
 Let $(f_i)_{i\in I}$ be a net  in $\Lip_0(N)$ which $w^*$-converges to 0, and let $\mu\in \cF(M)$,
 say $\mu=\lim_{k\to\infty }\mu_k$ in $\cF(M)$, and thus in $\Lip^*_0(M)$, where  $\mu_k$ is of the form
 $$\mu_k=\sum_{j=1}^{l_k} a_{(k,j)} \delta_{m(k,j)}, \text{ with $(a_{(k,j)})_{j=1}^{l_k}\subset \R$, and $(m_{(k,j)})_{j=1}^{l_k}\subset M$.}$$
 Then 
 $$\tilde\mu_k=\sum_{j=1}^{l_k} a_{(k,j)} \delta_{\varphi(m(k,j))}$$
 converges in $\cF(N)$ to some element $\tilde\mu\in \cF(N)$ with the property that
 $\la \tilde\mu, f\ra =\la \mu , f\circ \varphi\ra $, for $f\in \Lip_0(N)$. Indeed, note that for $k,k'\in \N$
 \begin{align*}
 \|\tilde\mu_k-\tilde\mu_{k'}\|_{\Lip^*}&= \sup_{f\in\Lip_0(N),\|f\|_L\le 1} \la f,\  \tilde\mu_k-\tilde\mu_{k'}\ra\\
&=\sup_{f\in\Lip_0(N), \|f\|_L\le 1} \la f\circ \varphi,\  \mu_k-\mu_{k'}\ra\\
&\le \lambda\sup_{g\in\Lip_0(M),\|g\|_L\le 1} \la g,  \mu_k-\mu_{k'}\ra= \|\mu_k-\mu_{k'}\|_{\Lip^*}. 
    \end{align*}  
    It follows therefore that 
    $$\lim_{i\in I} \la f_i\circ\varphi ,\mu\ra =\lim_{i\in I}  \la f_i, \tilde \mu\ra =0$$ 
    and thus we verified that $\varphi^\#$ is $w^*$-continuous. 
    
    It follows therefore that $\varphi^\#$ is the adjoint of an operator $\hat\varphi:\cF(M)\to \cF(N)$.
    
    Since for $f\in \Lip_0(N)$ and $m\in M$
    $$\la \hat\varphi(\delta_m), f\ra=\la \delta_m, \varphi^\#(f)\ra = f(\varphi(m))= \la \delta_{\varphi(m)}, f\ra,$$
    it follows that $\hat\varphi(\delta_m)=\delta_{\varphi(n)}$, which finishes our proof.
 \end{proof} 
 The following extension result of $\cF(M)$ is the reason why Godefroy and Kalton coined the name {\em Lipschitz free space over $M$} for $\cF(M)$. 
 \begin{prop}{\rm \cite{Weaver1999}*{Theorem 2.2.4}}\label{P:6.2.10} Let $M$ be a metric space and $X$ a Banach space, and let $L:M\to X$ be a Lipschitz function.
 Then a unique linear and bounded  {\em extension  } $\hat L: \cF(M)\to X$ of $L$ exists. This   means that 
 $\hat L(\delta_m)=L(m)$, for all $m\in\N$.
 
 Moreover we have in this case that $\|\hat L\|_{\cF(M)\to X}=\|L\|_L$.
 \end{prop} 
 \begin{proof} Since the $\delta_m$, $m\in M\setminus\{0\}$ are linearly independent as elements of  $\cF(M)$  we can at least extend $L$ linearly to $\spa(\delta_m: m\in M\setminus\{0\})$. 
 We denote this extension by $\tilde L$, 
 and we have to show that   $\tilde L$ extends to a bounded linear operator on $\cF(M)$, whose operator norm coincides with  
   $\| L\|_{\Lip}$. Let $\mu=\sum_{j=1}^n a_j \delta_{m_j}\in \spa(\delta_m: m\in M)$. Then there is some $x^*\in B_{X^*} $,  so that 
 $\la x^*, \tilde L(\mu)\ra =\|\tilde L(\mu) \|$. It follows that $x^*\circ L$ is in $\Lip_0(M)$ and $\|x^*\circ L\|_{\Lip}\le \|L\|_{\Lip},$
 and thus, 
 \begin{equation}\label{E:6.2.10.1}
 \|\tilde L(\mu)\|_X= \la x^*, \tilde L(\mu)\ra\le \|x^*\circ L\|_{\Lip}\cdot \|\mu\|_{\Lip^*}\le \Lip(L)\cdot \|\mu\|_{\Lip^*}.\end{equation}
  Thus, $\tilde L$ can be extended to a linear bounded operator $\hat L $ on all of $\cF(M)$. $\hat L$ is of course also Lipschitz with 
  $\Lip(\hat L)=\|\hat L\|_{\cF(M)\to X}$ and since $M$ isometrically embeds into $\cF$, it follows that
  $$\|\hat L\|_{\cF(M)\to X}=\Lip(\hat L)\ge \sup_{m,m'\in M, m\not=m'} \frac{ \|L(m)-L(m')\|}{d(m,m')}=\Lip(L),$$
  and thus, together with \eqref{E:6.2.10.1},   we deduce that $\|\hat L\|_{\cF(M)\to X}=\Lip(L)$.
 \end{proof}
 \begin{prop}\label{P:6.2.11} If $M$ is a metric space and $N\subset M$, then $\cF(N)$ is (in the natural way) a subspace of $\cF(M)$. 
 \end{prop}
 The proof of Proposition \ref{P:6.2.11} follows from the following extension result for  Lipschitz functions.
 \begin{lem}\label{L:6.2.12} Any Lipschitz function $f:N\to \R$ can be extended to a Lipschitz function $F:M\to \R$, by defining
 $$F(m)  =\inf_{n\in N}( f(n) +\|f\|_L d(n,m)), \text{ for $m\in M$}.$$
Moreover, it follows that  $\Lip(F)=\Lip(f)$.
 \end{lem}
 
 \begin{proof} Assume that $m,m'\in M$, and assume without loss of generality that $F(m)\le F(m')$. Let $\vp>0$ and choose $n\in N$ so that 
$F(m)\ge f(n) +\Lip(f)  d(n,m)-\vp$.

Then it follows 
\begin{align*}
0&\le F(m')-F(m)\\
   &\le      f(n) +\Lip(f) d(n,m')- \big(f(n)+\Lip(f) d(n,m')\big) +\vp\\
    & =\Lip(f) ( d(n,m) - d(n,m'))+\vp \le \Lip(f) d(m',m) +\vp,\end{align*}
thus, we deduce the claim if we let $\vp$ tend to $0$.
 \end{proof}

 \section{The Transportation cost norm}
\subsection{The Duality Theorem of Kantorovich}

 This section presents an intrinsic definition of the space $\cF(M)$ for a metric space $(M,d)$, \ie\ a definition which only uses the metric on $M$. 
 It represents $\cF(M)$ as a {\em Transportation Cost Space}.

Let $\mu=\sum_{j=1}^n a_j \delta_{x_j}\in \tilde\cF(M)$, $(a_j)_{j=1}^n\subset \R$ and $(x_j)_{j=1}^n \subset M\setminus\{0\}$.  
Since $\delta_0\equiv 0$ (as a functional acting on $\Lip_0(M)$), we can put $a_0=-\sum_{j=1}^n a_j$, and write $\mu$ as
$$\mu=a_0\delta_0+ \sum_{j=1}^na_j \delta_{x_j}$$
Thus, from now on, we define
$$\tilde\cF(M)= \Big\{ \sum_{j=1}^n a_j \delta_{x_j}: n\in \N, (a_i)_{j=1}^n\subset \R,\, (x_j)_{j=1}^n \subset M, \text{ with }\sum_{j=1}^n a_j=0\Big\}.$$

\begin{prop}\label{P:1.1.1} Every $\mu\in \tilde\cF(M)$ can be represented as
\begin{equation}\label{E:1.1.1.1}
\mu=\sum_{j=1}^l r_j (\delta_{x_j}-\delta_{y_j}), \text{ with $(x_j)_{j=1}^l,(y_j)_{j=1}^l\subset M, (r_j)_{j=1}^l \subset \R^+$}.
\end{equation}\end{prop}

\begin{defin}\label{D:1.2}
For $x,y\in M$ we call $\delta_x-\delta_y$, with $x\not=y$ a {\em molecule in $\tilde F(M)$} and \eqref{E:1.1.1.1} a {\em molecular representation of $\mu\in \tilde\cF(M)$.}
We will always assume in that case that $x_j\not=y_j$.
\end{defin}
\begin{proof}[Proof of Proposition \ref{P:1.1.1}]
We can write $\mu\in\tilde \cF(M)$ as
$$\mu=\sum_{i=1}^m a_i \delta_{x_i}-\sum_{j=1}^n b_j \delta_{y_j}$$
with $a_i,b_j>0$, $1\le i\le m$, $1\le j\le n$, and  $S:=\sum_{i=1}^m a_i =\sum_{j=1}^n b_j $, and $(x_i)_{i=1}^m$, $(y_j)_{j=1}^n\subset M$.

\begin{align*}\mu&=\frac1S\Big( \sum_{j=1}^n b_j\Big)\sum_{i=1}^m a_i \delta_{x_i}-\frac1S\Big(\sum_{i=1}^m a_i \Big)\sum_{j=1}^n b_j \delta_{y_j}\\
&=\frac1S \sum_{i=1}^m\sum_{j=1}^n a_ib_j\big(\delta_{x_i}-\delta_{y_j}\big).
\end{align*} 
\end{proof}
\begin{defin}\label{D:1.3}
Let $\mu=\cF(M)$ have the molecular representation
$$\mu=\sum_{j=1}^n r_j (\delta_{x_j}-\delta_{y_j}),$$
$r_j>0$, $x_j,y_j\in M$, for $j=1,2,\ldots n$, then we define
\begin{equation}\label{E:1.3.1} t\big((r_j)_{j=1}^n, (x_j)_{j=1}^n,(y_j)_{j=1}^n\big):= \sum_{j=1}^n r_jd(x_j,y_j),\end{equation}
and call it the {\em transportation costs of that representation}.
\end{defin}

\noindent{\bf Interpretation:}  Let us assume that transporting $a$ units of a product from $x$ to $y$ costs $a\cdot d(x,y)$.

Let $\mu=\mu^+-\mu^-\in \tilde \cF(M)$,
where  $\mu^+$ is the positive part  and a negative part $\mu^-$.
We interpret  $\mu^+$  as the  distribution of the surplus and 
$\mu^-$ as the distribution of the need of the product.

Then, a molecular  representation
$$\mu=\sum_{j=1}^l r_j (\delta_{x_j}-\delta_{y_j})$$
can be seen as a {\em transportation plan}  (and will be called as such)  to move $r_j$ units from $x_j$ to $y_j$ and thereby 
balancing the surplus with the need. For such a transportation plan, 
$t\big((r_j)_{j=1}^n, (x_j)_{j=1}^l(y_j)\big)$ represents the 
total transportation costs.

We define  
$$\|\mu\|_{\tc} =\inf\Big\{ t\big((r_j)_{j=1}^n, (x_j)_{j=1}^n(y_j)_{j=1}^n\big): \mu=\sum_{j=1}^n r_j(\delta_{x_j}-\delta_{y_j})\Big\}$$

 We will soon see that $\|\cdot\|_{\tc}$ is a norm.
We first want to show that the $\inf$ in the definition of $\|\cdot\|_{\tc}$ is attained for $\mu\in\cF(M)$.
To do that we need the following proposition
\begin{prop}\label{P:1.4}
Assume $\mu\in\tilde\cF(M) $  and 
$$\mu=\sum_{j=1}^n r_i(\delta_{x_j}-\delta_{y_j})$$
is a molecular representation.

Then, there exists a molecular  representation
$$\mu=\sum_{j=1}^{n'} r'_i(\delta_{x'_j}-\delta_{y'_j})$$
for which the sets $\{x'_j:j=1,2,\ldots, n'\}$ and 
$\{y'_j:j=1,2,\ldots n'\}$  are disjoint and 
$$t\big((r'_j)_{j=1}^n,(x'_j)_{j=1}^{n'},(y'_j)_{j=1}^{n'}\big)\le t\big((r_j)_{j=1}^n,(x_j)_{j=1}^n,(y_j)_{j=1}^n\big).$$
In this case, we call
$$\mu=\sum_{j=1}^{n'} r'_i(\delta_{x'_j}-\delta_{y'_j})$$
{\em a disjoint molecular representation}.
\end{prop}
\begin{proof} Assume  $\{x_j:j=1,2,\ldots, n\}$ and 
$\{y_j:j=1,2,\ldots n\}$ are not disjoint and without loss of generality $x_{n-1}=y_n$.
We write 
\begin{align*}
\mu=\sum_{j=1}^{n-2} r_j(\delta_{x_j}-\delta_{y_j}) + \underbrace{r_{n-1}(\delta_{x_{n-1}}-\delta_{y_{n-1}})+r_{n}(\delta_{x_{n}}-\delta_{y_{n}})}_{\nu}.
\end{align*}
Assume $r_{n-1}\ge r_n$ (similar argument if $r_{n-1}<r_n$) and write  $\nu$ 
as 
$$\nu= (r_{n-1} - r_n)(\delta_{x_{n-1}}-\delta_{y_{n-1}}) + r_n(\delta_{x_n} -\delta_{y_{n-1}}) $$
and note that 
\begin{align*}
(r_{n-1}-r_n)d(x_{n-1},y_{n-1})+ r_nd(x_n,y_{n-1})&= r_{n-1}d(x_{n-1},y_{n-1})+ r_n(d(x_n,y_{n-1})- d(x_{n-1},y_{n-1}))\\
&  = r_{n-1}d(x_{n-1},y_{n-1})+  r_l(d(x_l,y_{n-1})- d(y_{n},y_{n-1}))\\ &\le  r_{n-1}d(x_{n-1},y_{n-1})+  r_l(d(x_n,y_{n}).
\end{align*}
Thus $$\mu= \sum_{j=1}^{n-2} r_j(\delta_{x_j}-\delta_{y_j}) +  (r_{n-1} - r_n)(\delta_{x_{n-1}}-\delta_{y_{n-1}}) + r_n(\delta_{x_n} -\delta_{y_{n-1}})$$
is a molecular representation 
eliminating  $x_{n-1}=y_n$ in the intersection without increasing the transportation costs. 

We can therefore iterate this procedure until we arrive at a disjoint representation of $\mu$.
\end{proof}

\begin{rem} Let $$\mu=\sum_{j=1}^n r_j(\delta_{x_j}-\delta_{x_j})$$
be a disjoint molecular representation of $\mu$.

Then it follows from the {\em Jordan Decomposition Theorem} (which in the simple case of finite linear combinations of Dirac measures is trivial)
$$\mu^+= \sum_{j=1}^n r_j\delta_{x_j} \text{ is the  positive part and }\mu^-=\sum_{j=1}^n r_j\delta_{y_j} \text{ is  the negative part of $\mu$}.$$
Put
$$A^+=\supp(\mu^+)=\{ x\in M: \mu^+(x)>0\} \text{ and } A^-=\supp(\mu^-)=\{ x\in M: \mu^-(x)>0\}.$$
A disjoint  molecular decomposition is then always of the following form
$$\mu=\sum_{x\in A^+}\sum_{y\in A^-} \nu(x,y)(\delta_x-\delta_y).$$ 
For $x\in A^+$  and $y\in A^-$
$$ \mu(x)=\mu^+(x) =\sum_{y\in A^-} \nu(x,y)\text{ and } \mu^-(y)= \sum_{x\in A^+} \nu(x,y).$$
Thus $\nu$  (put $\nu(x,y)=0$ if $(x,y)\not= A^+\times A^-$) can be seen as a (positive) measure on $M^2$ whose {\em marginals} are $\mu^+$ and $\mu^-$.
Also, note that. 
$$\nu(M^2) =\sum_{x,y\in M^2} \nu(x,y) = \sum_{x\in M} \sum_{y\in M} \nu(x,y)=\sum_{x\in A^+} \sum_{y\in A^-} \nu(x,y) = \mu^+(M) =\mu^-(M).$$

Let $t(\nu)$ be the transportation costs of the representation
$$\mu=\sum_{x\in A^+}\sum_{y\in A^-} \nu(x,y)(\delta_x-\delta_y).$$
Then
$$t(\nu)=\sum_{x,y\in M} \nu(x,y) d(x,y)= \int_M\int_M d(x,y) \, d\nu(x,y).$$ 
and thus
\begin{align}\label{E:1.4.1}
 \|\mu\|_{\tc}&= \inf \left\{ \sum_{(x,y)\in M^2} \nu(x,y) d(x,y): \begin{matrix} \nu\text{ measure  on }M^2, \\
  \mu^+(x)=\sum_{y'\in M}\nu(x,y') \text{ for $x\in M$ and } \\ \mu^-(y)=\sum_{x'\in M}\nu(x',y)\text{ for $y\in M$} \end{matrix} \right\} \\
  &= \inf \left\{ \sum_{(x,y)\in M^2} \nu(x,y) d(x,y): \begin{matrix} \nu\text{ measure  on }M^2, \nu(M^2)=\mu^+(M)\\
  \supp(\nu)\subset \supp(\mu^+)\times \supp(\mu^-)\\
  \mu^+(x)=\sum_{y'\in M}\nu(x,y') \text{ for $x\in M$ and } \\ \mu^-(y)=\sum_{x'\in M}\nu(x',y)\text{ for $y\in M$} \end{matrix} \right\}.\notag
 \end{align}
\end{rem} 
From compactness, we therefore deduce that 

\begin{cor}\label{C:1.5} For $\mu\in \tilde\cF$ 
$$\|\mu\|_{\tc}=\inf\Big\{ t\big((r_j)_{j=1}^n, (x_j)_{j=1}^l(y_j)\big): \mu=\sum_{j=1}^n r_j(\delta_{x_j}-\delta_{y_j})\Big\}$$
is attained. 

We call  a representation of $\mu$  {\em optimal} 
if its transportation cost equals to $\|\mu\|_{\tc}$.
\end{cor} 
\begin{cor}\label{C:1.6} For $x,y\in M$
$$\|\delta_x-\delta_y\|_{\tc}=d(x,y).$$
\end{cor} 
\begin{proof} By above remark
$\|\mu\|_{\tc} =t(\nu)$ where $\nu$ is a measure on $M\times M$ with $\supp(\nu)=\{(x,y)\}$, and  thus $\nu=\delta_x-\delta_y$ is an optimal representation
\end{proof}

  \begin{thm}[Duality Theorem of Kantorovich \cite{Kantorovich1942}]\label{T:1.7} For $\mu\in\tilde\cF(M)$ it follows that 
  $\|\mu\|_\cF=\|\mu\|_{\tc}$.
  \end{thm}
    \begin{proof}
 It is easy to see that $\|\cdot\|_{\tc}$ is a semi norm on $\tilde\cF(M)$. 
 By Corollary \ref{C:1.6} $\|\delta_x-\delta_{y}\|_{\tc}=d(x,y)$, for $x,y\in M$.
 
\noindent{\bf Claim.} For every norm $\|\cdot\|$ on $\tilde \cF(M)$ with $\|\delta_x-\delta_y\| =d(x,y)$, for all $x,y\in M$, it follows that $\|\mu\|_{\tc}\ge \| \mu\|$ for all $\mu\in \tilde \cF(M)$.

Indeed, let $\mu=\sum_{x,y} \nu(x,y) (\delta_x-\delta_y)\in \tilde\cF(M)$ be a molecular representation.
Then 
$$t(\nu)= \sum_{x,y\in M} \nu(x,y) d(x,y) \ge \Big\|\sum_{x,y} \nu(x,y) (\delta_x-\delta_y)\Big\|.$$
So the claim follows by taking the infimum over all representations.
 
Since $\|\cdot \|_{\cF} $ is such a norm, it follows  that $\|\cdot\|_{\cF}\le \|\cdot\|_{\tc}$, and in particular that  $\|\cdot\|_{\tc}$ is also a norm.

Let $X$ be the completion of 
$\tilde\cF(M)$ with respect  to $\|\cdot\|_{\tc}$. 
 The map $L:M \to X$, $x\to \delta_x$ is an isometric embedding, which by the extension property of $\cF(M)$ can be extended to a linear operator 
 $\bar L:\cF(M)\to X$ with $\|\bar L\|_{\cF(M)\to X}=1$.
 
 This means that  $\|\mu \|_{\cF}\ge \|\mu \|_{\tc}$, thus  $\|\mu \|_{\cF}= \|\mu \|_{\tc}$, for $\mu\in\tilde\cF(M)$.
Therefore   $X=\cF(M)$ and $\|\cdot\|_{\tc}=\|\cdot \|_\cF$ the norms are the same.
\end{proof}
\begin{cor}\label{C:1.7b}Let $\mu\in \tilde\cF(M)$. Then a representation $\mu=\sum_{j=1}^n r_j (\delta_{x_j}-\delta_{y_j})$ is optimal 
if and only if there is an $f\in \Lip_0(M)$, with $\|f\|_\Lip=1$, for which
\begin{equation}\label{E:1.7b.1} f(x_j)-f(y_j)= d(x_j,y_j), \text{ for $j=1,2,\ldots, n$.}\end{equation}
\end{cor}
\begin{proof} If the representation  $\mu=\sum_{j=1}^n r_j (\delta_{x_j}-\delta_{y_j})$,with  $r_j>0$ and $x_j,y_j\in M$, for $j=1,2,\ldots,n$, is optimal, then it follows from the Hahn-Banach theorem and Theorem \ref{T:1.7}, that there
is an $f\in \Lip_0(M)$, $\|f\|_\Lip=1$ for which 
$$\int f \,d\mu = \sum_{j=1}^n r_j \big(f(x_j) -f(y_j)\big) =\|\mu\|_{tc}= \sum_{j=1}^n r_i d(x_j,(y_j) $$
since $\|f\|_\Lip=1$ it follows  that $ f(x_i) -f(y_i)\le d(x_j,(y_j)$, and, thus,  $ f(x_i) -f(y_i)= d(x_j,y_j)$, for all $j=1,2, \ldots n$.

Conversely, if \eqref{E:1.7b.1} holds for some $f\in \Lip_0(M)$, with $\|f\|_\Lip=1$, then by Theorem \ref{T:1.7} 
\begin{equation}\label{E:1.7b.2}
\int f \,d\mu = \sum_{j=1}^n r_j \big(f(x_j) -f(y_j)\big)\le \|\mu\|_\cF=\|\mu\|_{tc}.
\end{equation}
On the other hand, by definition of $\|\mu\|_{tc}$,
it follows that
 $$\|\mu\|_{tc}\le  \sum_{j=1}^n r_i d(x_j,(y_j)=  \sum_{j=1}^n r_i\big(f(x_j)-(y_j)\big),$$
and thus, the inequality in \eqref{E:1.7b.2} is an equality, and the representation is optimal. 
\end{proof} 
\subsection{The Extreme Points of $B_{\cF(M)}$}

\begin{defin} Let $\mu\in \cF(M)$ and 
let $$\mu=\sum_{j=1}^n r_j (\delta_{x_j} -\delta_{y_j})$$
be a molecular representation with  $r_j>0$, $j=1,2,\ldots n$.
We call a   sequence $(j_i)_{i=1}^l$ of pairwise distinct elements in $\{1,2,\ldots n\}$ {\em path } for the above representation if 
$x_{j_{i+1}}=y_{j_i}$, for $i=1,2,\ldots, l-1$, and we call it a {\em circle}, if
moreover $y_{j_l}=x_{j_1}$
\end{defin}
Not every optimal representation needs to be disjoint. Nevertheless, if it is not disjoint, something special has to happen.

\begin{prop}\label{P:1.9} If for $\mu\in\cF(M)$ the representation
\begin{equation}\label{E:1.9.1}
\mu=\sum_{j=1}^n r_j (\delta_{x_j} -\delta_{y_j}).\end{equation}
 is optimal,
then for any path $(j_i)_{i=1}^l$ it follows that 
$$d(x_{j_1}, y_{j_l})=\sum_{i=1}^l d(x_{j_i}, y_{j_i})= d(x_{i_1},y_{i_1})+\sum_{i=2}^l d(y_{j_{i-1}}, y_{j_i}).$$
In particular, it does not contain a circle.
\end{prop} 
Intuitively the claim is clear.
\begin{proof} Without loss of generality (after reordering), we can assume that $j_i=n-l+i$, for $i=1,2,\ldots l$. We put $\vp=\min\{r_j: n-l< j\le n\}$.
Then we write 
\begin{align*}
\mu&=\sum_{j=1}^{n-l} r_j(\delta_{x_{j}}-\delta_{y_{j}})+ \sum_{j=n-l+1}^n (r_j-\vp)(\delta_{x_{j}}-\delta_{y_{j}})+\vp\underbrace{\sum_{j=n-l+1}^n (\delta_{x_{j}}-\delta_{y_{j}})}_{=:\nu}\\
      &=\sum_{j=1}^{n-l} r_j(\delta_{x_{j}}-\delta_{y_{j}})+ \sum_{j=n-l+1}^n (r_j-\vp)(\delta_{x_{j}}-\delta_{y_{j}})+\vp(\delta_{x_{j_1}}-\delta_{y_{j_l}}).
      \end{align*}
From the optimality of the representation \eqref{E:1.9.1}, it follows that
      $$d(x_{j_1},y_{j_l})=\sum_{i=1}^l d(x_{j_i}, y_{j_i}).$$
\end{proof}
We deduce
\begin{cor}\label{C:1.10} Assume $x,y\in M$, $x\not=y$, and there is no $z\kin M\setminus\{x,y\}$ for which $d(x,y)\keq d(x,z)\kplus d(z,y)$.
Then every optimal representation of $\delta_x-\delta_y$ is disjoint, \ie\ by the remark after Proposition \ref{P:1.4}
it is of the form
$$\delta_x-\delta_y=\sum_{j=1}^n r_j (\delta_{x}-\delta_y), \text{with $\sum_{j=1}^n r_j =d(x,y)$}.$$
\end{cor}
\begin{proof} Note that a representation 
$$\delta_x-\delta_y= \sum_{j=1}^n r_j (\delta_{x_j}-\delta_{y_j}),$$
    which is not disjoint, would have  a path 
$(j_i)_{i=1}^l$, $l\ge 2$, with $x=x_{j_1}$ and $y=y_{j_l}$, but this would by the assumption on $x$ and $y$ and by Proposition \ref{P:1.9} mean that this representation is not optimal.
\end{proof}

The following result was shown in the general case (\ie\  in the case that $M$ is not finite) by Aliega and Preneck\'a in \cite{Aliaga_Pernecka2020}.
Recall that for a Banach space $E$, and  a subset $C\subset E$ an element $x\in C$  is called {\em extreme point of $C$} if 
$x$ cannot be written as   $x=\alpha y+ (1-\alpha) z$,with $0<\alpha<1$ and $y\not=z$, $z,y\in C\setminus \{x\}$.
In other words, if 
$$x=\sum_{j=1}^nx_j \alpha_j, \, 0<\alpha_j<1, \text{ and } x_j\in C,\, j=1,2\ldots n,\quad\text{ with }\sum_{j=1}^n \alpha_j=1,$$
then $x_1=x_2=\ldots =x_n=x$.

Recall (Krein-Milman Theorem): Every convex and compact subset $C$ of $E$ is the closed convex hull of its extreme points,

\begin{thm}\label{T:1.11} Let $(M, d)$ be a finite metric space, then 
$\mu\in B_{\cF(M)}$ is an extreme point if and only if 
$\mu$ is of the form
$$\mu=\frac{\delta_x-\delta_y}{d(x,y)}, \text{ with $x\not=y$, and there is no $z\kin M\setminus\{x,y\}$ for which $d(x,y)\keq d(x,z)\kplus d(z,y)$.}$$
\end{thm} 
\begin{proof} (We are using that $\cF(M)=\tilde \cF(M)$) Assume that $\mu\in B_{\cF(M)}$  is an extreme point, and  let  
$$\mu=\sum_{j=1}^l a_j\frac{\delta_{x_j}-\delta_{y_j}}{d(x_j,y_j)},$$
be its optimal representation.
 Thus it follows 
that 
$$1=\|\mu\|_{\cF}= \sum_{j=1}^l a_j.$$
Since $\mu$ is an extreme point it follows that $x_i=x_j=x$  and all $y_i=y_j=y$ for all $i,j\in\{1,2,\ldots, l\}$, and thus $\mu=\frac{\delta_x-\delta_y}{d(x,y)}$.
There cannot be a $z\in M$ so that $d(x,y)=d(x,z)+d(z,y)$, because otherwise we could 
write
$$\mu=\frac{d(x,z)}{d(x,y)}\frac{\delta_x-\delta_z}{d(x,z)} +\frac{d(y,z)}{d(x,y)}\frac{\delta_z-\delta_y}{d(y,z)}.$$
and note that this is also an optimal representation.

Conversely, assume that $x\not=y$ are in $M$, and that there is no $z\in M$ for which $d(x,y)=d(x,z)+d(z,y)$,
and assume we write 
$$\frac{\delta_x-\delta_y}{d(x,y)}=\alpha \mu+ (1-\alpha)\nu, \text{ with $\mu, \nu \in B_{\cF}$, and $0<\alpha<1$.}$$
Write the  optimal decompositions of $\mu$ and $\nu$ as 
$$\mu=\sum_{j=1}^l a_j \frac{\delta_{x_j}-\delta_{y_j}}{d(x_j,y_j)}\text{  and } \nu= \sum_{j=l+1}^{m+l} a_j  \frac{\delta_{x_j}-\delta_{y_j}}{d(x_j,y_j)},$$
with $a_j\ge 0$, $j=1,2\ldots,l+m$. By Corollary \ref{C:1.10} 
\begin{align*}
1=\Big\|\frac{\delta_x-\delta_y}{d(x,y)}\Big\|_{\cF}
 \le \alpha \|\mu\|_{\cF}+ (1-\alpha)  \|\nu\|_{\cF}
= \alpha \sum_{j=1}^{l} a_j+(1-\alpha) \sum_{j=l+1}^{l+m} a_j =1.
\end{align*}
This implies that 
$$\frac{\delta_x-\delta_y}{d(x,y)}=\sum_{j=1}^{m+l} b_j \frac{\delta_{x_j} -\delta_{y_j}}{d(x_j,y_j)} $$
with $b_j=\alpha a_j$, if $j=1,2\ldots,l$ and $b_j=(1-\alpha) a_j$, if $j=l+1,l+2,\ldots, l_{l+m}$, is an optimal 
representation of $\frac{\delta_x-\delta_y}{d(x,y)}$, and thus, from 
Corollary \ref{C:1.10} it follows that 
$x_j=x$ and $y_j=y$, for $j=1,2,\ldots,, m+l$. This implies that $\frac{\delta_x-\delta_y}{d(x,y)}$, is an extreme point of $B_{\cF(M)}$.
\end{proof} 
\subsection{Some Notational  Remarks}\label{ss:2.3}
In the literature there are several names for the space $\cF(M)$: Other than the {\em Lipschitz free space over $M$}, $\cF(M)$ is also called 
\begin{itemize}
\item {\em Transportation Cost Space},
\item {\em Wasserstein Space}  or more precisely  {\em Wasserstein 1-Space},
\item {\em Arens Eals} (denoted by $\AE$),
\item {\em Earthmover Space}.
\end{itemize}
Let us introduce some more notation 

Denote  the set of  measures on $M$, with finite support, by $\cM$ \ie
$$\cM(M)=\Big\{\sum_{j=1}^n a_j \delta_{x_j} : n\in\N,  a_j\in \R, x_j\in M, \text{ for $j=1,2,\ldots n$}\Big\}$$
(actually $\cM(M)=\tilde\cF(M)$).
Let $\cM^+(M)$ denote the positive measures and 
$\cP(M)$ denote  probabilities on $M$, with finite support.
For $\sigma,\tau\kin\cP(M)$ define the {\em Wasserstein distance of $\sigma$ and $\tau $} by
\begin{align*}
d_\Wa(\sigma,\tau)=\|\sigma-\tau\|_\cF= \|\sigma-\tau\|_\tc
\end{align*}
Thus if we  let $\mu=\sigma-\tau$, it follows by the Remark after Proposition \ref{P:1.4} that
 $$d_\Wa(\sigma,\tau)= \inf \left\{ \sum_{(x,y)\in M^2} \nu(x,y) d(x,y): \begin{matrix} \nu\in \cM(M^2), \\
  \mu^+(x)=\sum_{y'\in M}\nu(x,y') \text{ for $x\in M$ and } \\ \mu^-(y)=\sum_{x'\in M}\nu(x',y)\text{ for $y\in M$} \end{matrix} \right\}$$
We claim that 
$$d_\Wa(\sigma,\tau)= \inf \left\{ \sum_{(x,y)\in M^2} \pi(x,y) d(x,y): \begin{matrix} \\ \pi\in \cP(M^2), \\
  \sigma(x)=\sum_{y'\in M} \pi(x,y') \text{ for $x\in M$ and } \\ \tau(y)=\sum_{x'\in M}\pi(x',y)\text{ for $y\in M$} \end{matrix} \right\}.$$
(which is the usual definition of the Wasserstein distance).

Indeed 
$\sigma-\tau=\mu=\mu^+-\mu^-$
and thus, for every $x\in M$
$$\rho(x):=\sigma(x)-\mu^+(x)=\tau(x)-\mu^-(x)$$
Let $\nu\in \cM^+(M^2)$ be such that
\begin{equation}\label{E:1.1}
 \mu^+(x)=\sum_{y'\in M}\nu(x,y'), \text{ for $x\in M$ and }  \mu^-(y)=\sum_{x'\in M}\nu(x',y),\text{ for $y\in M$}. \end{equation}
Then define  $\pi= \nu+\sum_{x\in M} \delta_{(x,x)} \rho(x)$
                                        and note that 
$$ \sigma(x)=\sum_{y'\in M}\pi(x,y'), \text{ for $x\in M$ and }  \tau(y)=\sum_{x'\in M}\pi(x',y),\text{ for $y\in M$},$$ 
and thus $\pi\in \cP(M^2)$.
Moreover, we have
\begin{equation}\label{E:1.2}\sum_{(x,y)\in M^2} \nu(x,y) d(x,y)=\sum_{(x,y)\in M^2} \pi(x,y) d(x,y), \end{equation}

Similarly it follows that if $\pi\in \cP(M^2)$ satisfies \eqref{E:1.2} then 
$\nu=\pi-\sum_{x\in M} \delta_{(x,x)}\pi(x,x)$ satisfies \eqref{E:1.1}.

We define for $\sigma,\tau\in \cP(M)$
the {\em Transition Probabilities from $\sigma$ to $\tau$} 
$$\cP(\sigma,\tau)=\Big\{ \pi\in \cP(M^2):\sigma(x)=\sum_{y'\in M}\pi(x,y'),  \tau(y)=\sum_{x'\in M}\pi(x',y),\text{ for $x,y\in M$}\Big\}$$
and can rewrite $d_\Wa(\sigma,\tau)$
as 
\begin{equation}\label{E:1.3}
d_\Wa(\sigma,\tau)=\min\Big\{ \sum_{x,y\in M} \pi(x,y)d(x,y) : \pi\in\cP(\sigma, \tau)\Big\} .
\end{equation}
\begin{defin}  We call  $\cP(M)$ together with the metric $d_\Wa(\cdot,\cdot)$ {\em Wasserstein space}, and denote it by $\Wa(M)$.
\end{defin}

More generally if $\mu_1,\mu_2\in \cM^+(M)$, with $\mu_1(M)=\mu_2(M)$,  we put 
$$\cM^+(\mu_1,\mu_2)=\Big\{ \nu\in \cM^+(M^2),:\mu_1(x)=\sum_{y'\in M}\nu(x,y'),  \mu_2(y)=\sum_{x'\in M}\nu(x',y),\text{ for $x,y\kin M$}\Big\}$$
and conclude that 
\begin{equation}\label{E:1.4}
\|\mu_1-\mu_2\|_\cF= \min\Big\{ \underbrace{\sum_{x,y\in M} \nu(x,y)d(x,y)}_{=\int_{M^2} d(x,y) \, d\nu(x,y)} : \nu\in\cM^+(\mu_1, \mu_2)\Big\}.
\end{equation} 

\subsection{Uniform Distibutions}
   \medskip  
    For $A\subset M$ we denote the uniform distribution  on $A$  by $\mu_A$,\ie\ 
    $\mu_A(x)= \frac1{|A|}\chi_A(x)$.

    \begin{prop}\label{P:1.12}
    If $A,B\subset M$ with $n=|A|=|B|$, then  there exist a bijection $f:A\to B$
    So that 
    $$d_\Wa(\mu_A,\mu_B)=\frac1n \sum_{x\in M}  d(x,f(x))$$
    In other words, the representation
    $$\mu_A-\mu_B=\frac1n \sum_{x\in A} \big(\delta_x -\delta_{f(x)}\big)$$
    is optimal.
    \end{prop}

   The proof is a Corollary of the following theorem by   Birkhoff (see appendix for a proof)
   
   \begin{thm}\label{T:1.13}(Birkhoff)
   Assume $n\in\N$ and  that $A=(a_{i,j})_{i,j=1}^n$ is a doubly stochastic matrix, \ie
   \begin{align*}
   &0\le a_{i,j}\le 1\text{ for all $1\le i,j\le n$,}\\
   &\sum_{j=1}^n  a_{i,j}=1\text{ for $i=1,2,\ldots, n$,}\text{ and }
   \sum_{i=1}^n  a_{i,j}=1\text{ for $j=1,2,\ldots, n$.}
   \end{align*}
   Then $A$ is a convex combination of permutation matrices, \ie matrices which have in each 
   row and each column exactly one entry whose value is  $1$ and vanish elsewhere.
   \end{thm}
     
   \begin{proof}[Proof of Proposition \ref{P:1.12}] Let 
    $A=\{x_1,x_2,\ldots x_n\}$ and $B=\{y_1,y_2,\ldots, y_n\}$. 
   We note that for every $\pi\in \cP(\sigma,\tau)$ 
   the matrix\\ $M=(n \pi(x_i,y_j) : 1\le i,j\le n)$ is a doubly stochastic matrix
   (since $\sum_{x\in A} \pi(x,y) =\tau(y)=\frac1{|B|}=\frac1{|A|}=\sigma(x)=\sum_{y\in B} \pi(x,y)$).
   Thus by \eqref{E:1.3}
   \begin{align*} d_\Wa(\mu_A,\mu_B)=\frac1n \min \Big \{ \sum_{i,j=1}^n M_{i,j} d(x_i,y_j) : M\in \DS_n\Big\}
   \end{align*}
   Since the map
   $$\DS\to [0,\infty), \quad M\mapsto  \sum_{i,j=1}^n M_{i,j} d(x_i,y_j) $$
   is linear, it achieves its minimum on an extreme point; our claim follows from Theorem \ref{T:1.13} 
      \end{proof}
 \vfill\eject
\section{Embeddings of Transportation cost  spaces over trees into $L_1$}
\subsection{Distortion}
 Let $(M,d)$ and $(M',d')$ be two metric spaces.
For $f: M\to M'$ the {\em  distortion of $f$ } is defined by
$$\dist(f)=\sup_{x\not=y, x,y\in M} \underbrace{\frac{d'(f(x), f(y))}{d(x,y)}}_{\Lip(f)}\sup_{x\not=y, x,y\in M}\underbrace{\frac{d(x,y)}{d'(f(x), f(y))}}_{\Lip(f^{-1})} $$
where $f^{-1}$ is defined on $f(M)$, if $f$ is injective, and  otherwise $\dist(f):=\infty$.

We define  the {\em $M'$- distortion of  $M$} by
$$c_{M'}(M) =\inf \big\{ \dist(f)\vert \,  f:M\to M'\big\}.$$

Let $\cM'$ be a family of metric spaces. We define the {\em $\cM'$-distortion of $M$ } by
$$c_{\cM'}(M)=\inf_{M'\in \cM'} c_{M'}(M).$$

The main question we want to address:  

\begin{prob}
Let $(M,d)$  be a finite metric space.
Find upper and lower  estimates of 
\begin{itemize}
\item $c_{(\ell_1^n, n\in\N)}(\cF(M))$, 
 \item  $c_{(\ell_1^n, n\in\N)}(\Wa(M))$,
 \item  $ \inf\big\{ \|T\|_{\cF(M)\to L_1}\cdot \|T^{-1}\|_{T(L_1)\to \cF(M)} : T: \cF\to L_1 \text{ linear and bounded}\big\}$.
 \end{itemize}
\end{prob}

Our families of metric spaces are usually closed under scaling, \ie\ if $(M', d')\in \cM'$ and $\lambda>0$, then
also  $(\cM',\lambda\!\cdot\!d')\in \cM'$. In that case
$$c_{\cM'}(M)= \inf_{M'\in \cM'}  \{ \|f\|_\Lip: f: M\to M' \text{ is expansive}\}$$ 
where $f: (M,d)\to (M',d') $ is called {\em expansive}, if
$$d'(f(x),f(z))\ge d(x,z), \text{ for $x,z\in M$.}$$
\subsection{Lipschitz embeddings  of $\Wa(M)$ into $L_1$ imply  linear embeddings of $\cF(M)$ into $\ell_1(N)$}
Throughout this subsection, we assume that $(M,d)$ is a  finite metric space and let $n=|M|$.

\begin{thm}\label{T:2.2}
 Assume that 
$$ F: \cP(M)\to L_1[0,1]$$
has the property that for some $L\ge 1$
\begin{equation*}d_\Wa(\sigma,\tau)\le \|F(\sigma)-F(\tau)\|_1\le L d_\Wa(\sigma,\tau) \text{ for $\sigma,\tau\in \cP(M)$.}
\end{equation*}
Then, there exists a Lipschitz map
$$H: \cF(M)\to L_1[0,1]$$
for which 
$$\|\mu-\nu\|_\cF\le \|H(\mu)-H(\nu)\|_1\le 3 L\|\mu-\nu\|_\cF,  \text{ for $\mu,\nu\in \cF(M)$}.$$
\end{thm} 
\begin{rem} Using more sophisticated tools, we can obtain a linear bounded operator
$$T: \cF(M)\to L_1[0,1]$$
for which 
$$\|\mu\|_\cF\le \|T(\mu)\|_1\le  L\|\mu\|_1.$$
Here, we are following a more elementary but also more technical proof by Naor and Schechtman
\cite{NaorSchechtman2007}.
\end{rem} 
\begin{proof} After scaling we can assume that $d(u,v)\ge 1$ for all $u\not=v$ in $M$, and secondly we can 
assume that the image of the uniform distribution, $\mu_0= \frac1n\sum_{x\in M}\delta_x$, under $F$ vanishes.
Note that for $\mu\in \cF(M)$  
$$\|\mu\|_\infty=\max_{x\in M} |\mu(x)|\le \|\mu\|_{\cF}.$$
Indeed, if $\nu\in \cM^+(\mu^+,\mu^-)$, then 
\begin{align*}
\int_{M\times M}d(x,y) \,d \nu(x,y)&\ge \int _{M\times M}1 \,d \nu(x,y) = \int_{M} 1\,d\mu^+(x) =  \mu^+(M)
                                  = \mu^-(M) \ge \|\mu\|_\infty,
                                  \end{align*}
                                  and thus, taking the infimum over all $\nu\in \cM^+(\mu^+,\mu^-)$, we deduce our claim.

We put 
$$B=\{\mu\in\cF(M), \|\mu\|_\infty \le 1\}. $$

$$\Psi: B\to \cP(M),\quad  \mu\mapsto \sum_{x\in M}\frac{1+\mu(x)}n\delta_x\in\cP(M).$$
For any $f\in \Lip_0(M)$, $\mu,\nu\in B$
$$\int_M f(x) d(n\Psi(\mu)-n\Psi(\nu))=\sum_{x\in M} f(x) (\mu(x)-\nu(x))= \int_M f(x) d(\mu-\nu),$$
and thus 
$$\|\mu-\nu\|_{\cF(M)}= n\|\Psi(\mu)-\Psi(\nu)\|_{\cF(M)}=n d_\Wa\big(\Psi(\mu),\Psi(\nu)\big).$$
Then define 
$$h: B\to L_1[0,1] \quad\mu\mapsto n F\circ \Psi(\mu).$$
Then $h(0)=0$ (because the image of the uniform distribution under $F$ was assumed to vanish) and for $\mu,\nu\in B$
\begin{align}\label{E:2.2.1}
\|h(\mu)&-h(\nu)\|_1\\ 
&=n \|F(\Psi(\mu))-F(\Psi(\nu))\|_1\notag\\
&=n\Big\|F\Big(\sum_{x\in M}\frac1n (1+\mu(x))\delta_x\Big)-F\Big(\sum_{x\in M} \frac1n(1+\nu(x))\delta_x\Big)\Big\|_1\notag\\
&\ge n\Big\| \sum_{x\in M} \frac1n(1+\mu(x))\delta_x-\sum_{x\in M}\frac1n (1+\nu(x))\delta_x\Big\|_\cF\quad (\text{expansiveness)}\notag\\
&= \|\mu-\nu\|_\cF,  \notag
\end{align}
and
\begin{align}\label{E:2.2.2}
\|h(\mu)&-h(\nu)\|_1\\
&=n\Big\|F\Big(\sum_{x\in M}\frac1n (1+\mu(x))\delta_x\Big)-F\Big(\sum_{x\in M} \frac1n(1+\nu(x))\delta_x\Big)\Big\|_1\notag\\
&\le n  Ld_\Wa\Big(\sum_{x\in M}\frac1n (1+\mu(x))\delta_x,\sum_{x\in M}\frac1n (1+\nu(x))\delta_x\Big)
= \|\mu-\nu\|_\cF.
\end{align}

We define   $\chi(f): [0,1] \times \R\to \{-1,0,1\}$, for  $f\in L_1[0,1]$,  by
$$\chi(f)(s,t) = \sign\big(f(s)\big) 1_{[0,|f(s)|]}(t) = \begin{cases}  1 &\text{if $f(s)>0$ and $0\le t\le f(s)$}, \\
                                                                                                    -1 &\text{if $f(s)<0$ and $0\le t\le- f(s)$},\\
                                                                                                    0 &\text{else.}
\end{cases}$$
Note that (by Fubini's Theorem)

\begin{equation}\label{E:2.2.3}\|\chi(f)-\chi(g)\|_{L_1([0,1]\times \R)}=\|f-g\|_{L_1[0,1] }, \text{  for $f,g\in L_1[0,1]$}.
\end{equation} Indeed,
note that  $f(s)\ge g(s) \iff \chi(f)(s,t)\ge \chi(g)(s,t)$ and  that, since $\chi(f)$ and $\chi(g)$ are $-1,0,1$ valued, we have 
\begin{align*}
\int_0^1\int_{-\infty}^\infty & |\chi(f)(s,t)-\chi(g)(s,t)|\, dt ds\\
&=\int_{\{s: f(s)>g(s) \}}\int_{-\infty}^{+\infty} \chi(f)(s,t)-\chi(g)(s,t)\,dtds \\
  &\qquad+  \int_{\{s: f(s)<g(s) \}} \int_{-\infty}^{+\infty}\chi(g)(s,t)-\chi(f)(s,t) \,dtds \\
  &=\int_{\{s: f(s)>g(s) \}} f(s)-g(s) ds+\int_{\{s: f(s)<g(s) \}}g(s)-f(s) ds =\|f-g\|_{L_1[0,1]}.
\end{align*}
Now we put  (recall that $\mu/\|\mu\|_\cF\in B$, for $\mu\in\cF(M)$)
$$H:\cF(M)\to L_1([0,1]\times \R),\quad \mu\mapsto \|\mu\|_\cF \cdot\chi\circ  h \Big( \frac{\mu}{\|\mu\|_\cF}\Big) \text{if $\mu\not=0$, and 
$H(0)=0$.}$$
It follows  for $\mu,\nu\in \cF(M)$, with $\|\mu\|_\cF\ge \|\nu\|_\cF$ (w.l.o.g.), that
\begin{align*}
\big| H(\mu)-H(\nu)\big|&= \Big| \|\mu\|_\cF\cdot \chi\circ  h \Big( \frac{\mu}{\|\mu\|_\cF}\Big)- \|\nu\|_\cF \cdot\chi\circ  h \Big( \frac{\nu}{\|\nu\|_\cF}\Big)\Big|\\
&=\Big|\|\nu\|_\cF \Big(\chi\circ  h \Big( \frac{\mu}{\|\mu\|_\cF}\Big)-\chi\circ  h \Big( \frac{\nu}{\|\nu\|_\cF}\Big)\Big)\\
 &\qquad\qquad\qquad+(\|\mu\|_\cF-\|\nu\|_\cF)\cdot\chi\circ  h \Big( \frac{\mu}{\|\mu\|_\cF}\Big)\Big|\\
&= \|\nu\|_\cF \Big|\chi\circ  h \Big( \frac{\mu}{\|\mu\|_\cF}\Big)-\chi\circ  h \Big( \frac{\nu}{\|\nu\|_\cF}\Big)\Big|\\
 &\qquad\qquad\qquad+(\|\mu\|_\cF-\|\nu\|_\cF)\Big|\chi\circ  h \Big( \frac{\mu}{\|\mu\|_\cF}\Big)\Big|\\
 &\Big(\text{Check possible values  of  $\chi\circ h\Big( \frac{\mu}{\|\mu\|_\cF}\Big), \chi\circ h\Big( \frac{\nu}{\|\nu\|_\cF}\Big)\kin\{-1,0,1\}$}\Big).
\end{align*} 
Thus
\begin{align*}
\big\|H(\mu)-H(\nu)\big\|_{L_1([0,1]\times\R)}&= 
\|\nu\|_\cF \Big\|\chi\circ  h \Big( \frac{\mu}{\|\mu\|_\cF}\Big)-\chi\circ  h \Big( \frac{\nu}{\|\nu\|_\cF}\Big)\Big\|_1\\
 &\qquad\qquad\qquad+(\|\mu\|_\cF-\|\nu\|_\cF)\Big\|  h \Big( \frac{\mu}{\|\mu\|_\cF}\Big)\Big\|_1\\
&= \|\nu\|_\cF \Big\|  h \Big( \frac{\mu}{\|\mu\|_\cF}\Big)- h \Big( \frac{\nu}{\|\nu\|_\cF}\Big)\Big\|_1 \text{ (By \eqref{E:2.2.3})}\\
 &\qquad\qquad\qquad+(\|\mu\|_\cF-\|\nu\|_\cF)\Big\|  h \Big( \frac{\mu}{\|\mu\|_\cF}\Big)\Big\|_1\\
 &\ge \|\nu\|_\cF \Big\|\frac{\mu}{\|\mu\|_\cF}-  \frac{\nu}{\|\nu\|_\cF}\Big\|_\cF+\|\mu\|_\cF-\|\nu\|_\cF \text{ (By \eqref{E:2.2.1})}\\
 &\ge \|\nu-\mu\|_\cF-\Big\|\mu\Big(1-\frac{\|\nu\|_{\cF}}{\|\mu\|_\cF}\Big)\Big\|+\|\mu\|_\cF-\|\nu\|_\cF.
 =\|\mu-\nu\|_\cF
 \end{align*}
We also get 
\begin{align*}
\big\|H(\mu)-H(\nu)\big\|_{L_1([0,1]\times\R)}&= 
\|\nu\|_\cF \Big\|  h \Big( \frac{\mu}{\|\mu\|_\cF}\Big)- h \Big( \frac{\nu}{\|\nu\|_\cF}\Big)\Big\|_1 \\
 &\qquad\qquad\qquad+(\|\mu\|_\cF-\|\nu\|_\cF)\Big\|  h \Big( \frac{\mu}{\|\mu\|_\cF}\Big)\Big\|_1\\
 &\le L \|\nu\|_\cF  \Big\| \frac{\mu}{\|\mu\|_\cF}-\frac{\nu}{\|\nu\|_\cF}\Big\|_1 \\
&\qquad\qquad\qquad+L(\|\mu\|_\cF-\|\nu\|_\cF)  \text{ (By \eqref{E:2.2.2}) }\\
&=L\Big\|\frac{\|\nu\|_\cF}{\|\mu\|_\cF} \mu-\nu\Big\|_\cF+ L(\|\mu\|_\cF-\|\nu\|_\cF)\\
&\le L\Big(\Big\|\frac{\|\nu\|_\cF}{\|\mu\|_\cF}\mu-\mu\Big\|_\cF+\|\mu-\nu\|_\cF\Big) + L(\|\mu-\nu\|_\cF)\\
&=L  \Big( {\|\mu\|_\cF} -  {\|\nu\|_\cF} \Big)+   2L(\|\mu-\nu\|_\cF)\le 3 L(\|\mu-\nu\|_\cF).
\end{align*}
Thus $H$ is  a bi- Lipschitz map from $\cF(M)$ to $L_1([0,1]\times \R) \equiv L_1[0,1]$ of distortion not larger than $3L$.
\end{proof}
\begin{rem} If $F:\cP(M)\to L_1[0,1]$ is Lipschitz, we could perturb $F$ a bit, to a map $\tilde F$ having a finite-dimensional image and almost the 
same distortion.
If we could then produce a Lipschitz map  $H:\cF(M)\to L_1[0,1]$, which also has a finite-dimensional image we would only need Rademacher's Theorem 
to linearize $H$,
\end{rem}

The following result is a generalization of Rademacher's Theorem by Heinrich and Mankiewicz.

Let $X$ and $Z$ be   Banach spaces
and let  $f: X\to Z^*$ be a Lipschitz map. We say that $f$ is $w^*$ differentiable at a point $x_0$ if for all
$x\in X$
$$(D^*f)_{x_0} (x)= w^*-\lim_{\lambda\to 0} \frac{f(x_0+\lambda x)-f(x_0)}{\lambda} \text{  for all $x\in X$ exists.}$$
We call $(D^*f)_{x_0}$ the $w^*$-derivative of $f$ at $x_0$.
\begin{thm}\label{T:2.3} \cite{heinrichmankiewicz1982}*{Theorem 3.2}
Let $X$ and $Z$ be Banach spaces, $Z$ being separable,  and $f: X\to Z^*$ be a Lipschitz map. Then 
there is a dense set $D\subset X$ so that for all $x_0\in D$
the $w^*$ derivative 
$(D^*f)_{x_0}$ exists.

Moreover:
\begin{enumerate} 
\item For every $x\in D$, $(D^*f)_{x_0}$ is bounded linear operator from $X$ to $Z^*$ and $\|(D^*f)_{x_0}\|\le \|f\|_\Lip$.
\item If, moreover, $f$ is bi-Lipschitz, then $(D^*f)_{x_0}$ is an isomorphic embedding and
 $\|(D^*f)_{x_0}^{-1}\|$ $\le \|f^{-1}\|_\Lip$.
\end{enumerate}
\end{thm}

\begin{cor}\label{C:2.4}   Assume that 
$$ F: \cP(M)\to L_1[0,1]$$
has the property that 
\begin{equation}\label{E:2.4.1}d_\Wa(\sigma,\tau)\le \|F(\sigma)-F(\tau)\|_1\le Ld_\Wa(\sigma,\tau), \text{ for $\sigma,\tau\in \cP(M)$}
\end{equation}
and let $\vp>0$

Then there exists an $N\in\N$, and an linear embedding $T$ of $\cF(M)$ into $\ell_1^N$ so that 
$$\|\mu\|_\cF\le \|T(\mu)\|_1\le (3L+\vp)\|\mu\|_\cF.$$

\end{cor}
\begin{proof} Let $H: \cF(M)\to L_1[0,1]$ be defined as in Theorem \ref{T:2.2}. Since $L_1[0,1]$ is isometrically a subspace of $C^*[0,1]$
we can apply  Theorem \ref{T:2.3} and obtain an isomorphic embedding $S$  of $\cF(M)$ into $C^*[0,1]$, with 
$$\|\mu\|_\cF\le \|S(\mu)\|_1\le 3L\|\mu\|_\cF.$$
Since $\cF(M)$ is finite-dimensional $S(\cF(M))$ is also finite-dimensional.  $C^*[0,1]$ is  a $\cL_1$-space of constant 1, which means
that for every finite dimensional subspace $F$ of $C^*[0,1]$ there is a finite dimensional subspace $G$ of $C^*[0,1]$ which contains $F$,
and which is $(1+\vp)$-isomorphic to $\ell_1^N$, for some $N$ and $(1+\vp)$-complemented in $C^*[0,1]$ (the complementation is not needed).
We deduce from this our claim.
\end{proof}
\begin{exc} Prove that $C^*[0,1]$ is a $\cL_1$-space. 
\end{exc}

\subsection{Geodesic Graphs}

An undirected  graph $G$ is a pair $G=(V(G),E(G))$  with $$E(G)\subset [V(G)]^2=\{ e\subset V(G): |e|=2\}.$$
For $v$ we call
$$\deg(v)=\big\{ e\in E(G): v\in e\big\}$$
the {\em degree of $v$.}

A {\em walk}  in  a graph $G$ is a graph $W=(V(W),E(W))$ with $V(W)\subset V(G)$, and $E(W)\subset E(G)$,
and $V(W)$ can be ordered into $\{x_j:j=0,1,2,\ldots n\}$  (where  the $x_j$ are not necessarily distinct) so that 
$E(W)=\big\{ \{x_{j-1}, x_j\}: j=1,2\ldots n\}$. In that case we call $W$ {\em a walk from $x_0$ to $x_n$}, and also write $W=(x_j)_{j=0}^n$.

In the case that $x_j\not= x_i$, if $i\not=j$ in $\{0,1,2,\ldots,n\}$ we call $W$ a path.
If $x_0=x_n$, $x_j\not= x_i$,  unless $\{i,j\}=\{0,n\}$, we call $W$ a {\em cycle}.

We call a graph $G=(V(G),E(G))$ {\em connected } if for each $u,v\in V(G)$ there is a walk (and thus also a path) from $u$ to $v$.
A connected graph that does not contain a cycle is called a {\em tree}. In that case, a unique path exists between 
any two vertices $u$ and $v$. We denote that unique path by $[u,v]_G$.

\begin{prop}\label{P:2.1} Let $G=(V(G),E(G))$. The following statements are equivalent:
\begin{enumerate}
\item $G$ is a tree,
\item $G$ is minimal  connected graph, \ie\ for every $e\in E(G)$, the graph $G'=(V(G), E(G)\setminus\{e\})$ is not connected,
\item $G$ is a maximal graph without a cycle, \ie\ for every $e\in[V(G)]^2\setminus E(G)$, the graph $\tilde G=(V(G), E(G)\bigcup \{e\})$ has a cycle,
\end{enumerate}
and in the case that $n=|V(G)|<\infty$
\begin{enumerate}
\item[(4)] $G$ is connected and $|E(G)|=n-1$,
\end{enumerate}
\end{prop}
\begin{proof} Exercise.
\end{proof} 

\begin{defn} Let $T=(V(T),E(T))$ be a tree we call $v\in V(T)$ a {\em leaf of $T$} if $\deg(v)=1$,

\end{defn}
\begin{exc} Every finite tree has leaves.
\end{exc} 

If $G$ is a connected graph and  $d_G$ is a metric on $V(G)$, we call $d_G$ a {\em geodesic metric on $G$} if 
$$d_G(u,v)=  \min\{ \length_{d_G}(p): p\text{ is a path from $u$ to $v$}\},\text{ for $u,v\in V$,}$$
where for a path $p=(x_j)_{j=0}^n$ in $G$ we define the length of $p$ by
$$\length_{d_G}(p)=\sum_{j=1}^n d_G(x_{j-1},x_j).$$
In that case we call the pair $(G,d_G)$ a {\em geodesic graph}.
For $e=\{u,v\}\in E(G)$ we put $d_G(e)=d_G(u,v)$. If $G$ is a cycle, a path, or a tree, we call it a {\em geodesic cycle},
 {\em geodesic path}, or a {\em geodesic tree}, respectively.

Assume that $w:E(G)\to \R^+$ is a function. 
Define for  $u,v\in V(G)$
$$d_G(u,v):=\min\Big\{ \sum_{j=1}^n w(\{x_{i-1},x_i\}) : (x_j)_{j=0}^n \text{ is a path from $u$ to $v$}\Big\}.$$
Then $d_G$ is a geodesic metric on $G$, and we call it the {\em metric generated by the weight function} $w$.

This is why geodesic graphs are often referred to as {\em weighted graphs}.

\subsection{Transportation cost spaces over geodesic trees}

Let $(M,d)$ be a metric space. Recall the following notation from \ref{ss:2.3}:
for $\mu=\mu^+-\mu^-\in\tilde\cF(M)$, 
$$\cM^+(\mu_1,\mu_2)=\Big\{ \nu\in \cM^+(M^2):\mu_1(x)=\sum_{y'\in M}\nu(x,y'),  \mu_2(y)=\sum_{x'\in M}\nu(x',y),\text{ for $x,y\in M$}\Big\}.$$
We showed
\begin{equation*}
\|\mu\|_\cF= \|\mu\|_{\tc}=\min\Big\{ \sum_{x,y\in M} \underbrace{d(x,y)\, d\nu(x,y)}_{=\int_{M^2} d(x,y)\,\nu(x,y)}  : \nu\in\cM^+(\mu^+, \mu^-)\Big\}.
\end{equation*}

Throughout this subsection $T=(V(T), E(T))$ is a, possibly infinite, tree, and $d_T$ is a geodesic distance.
 Let $w:E(T)\to (0,\infty)$, $e\mapsto d(e)$ be the corresponding weight function.
  
 Our goal is to  show that 
$\cF(V(T),d_T)$ is isometrically isomorphic to the {\em weighted $\ell_1$-space}  $\ell_1(E(T), w)$, with the norm
$$\|x\|_1=\sum_{e\in E(T)} |x_e| w(e) \text{ for $x=(x_e:e\in E(T))\subset \R$}.$$

We first introduce some notation. 

Recall that for $u,v\in V(T)$, $[u,v]_T$ denotes the unique path from $u$ to $v$.
We choose a {\em  root} $v_0\in V$ and note that for any $e=\{u,v\}\in E(T)$ 
$$\text{Either $u\in [v_0,v]_T$ or $v\in[v_0,u]_T$}.$$
We assign to $e\in \{u,v\}\in E(T)$ an orientation, and choose $(u,v)$ iff  $u\in [v_0,v]_T$, which is equivalent to $d(v_0,u)< d(v_0,v)$
and define
$$E_d(T)=\big\{ (u,v): \{u,v\}\in E(T)\text{ and }d(v_0,u)< d(v_0,v\big\},$$
which means that if $e=\{u,v\}\in E(T)$, and $v$ is further awy from $v_0$ than $u$ then $(u,v)$ is the  orientation of $e$.
For $e=(u,v)\in E_d(T)$ we put $e^-=u$ and $e^+=v$.
We define a partial order in $V(T)$ by
$$u\preceq v:\iff  [v_0,u]_T\subset [v_0,v]_T.$$

From the uniqueness of paths in $T$, we deduce the following facts
\begin{exc} 
\begin{enumerate}
\item For every $v\in V(T)$,
we have $\{u\in V(t): u\preceq v\}=[v_0,v]_T$ , and thus is linearly ordered with respect to $\preceq$.
\item If $v\in V(T)$, $v\not=v_0$,  there is a unique
$e\in E(T)$ with $e^+=v$. We call, in that case, $e^-$ the {\em immediate predecessor of $v$.}

\item (Tripot property of trees) If $u,v\in V(T)$ then there exists a (unique) $z\in V(T)$ for which $[v_0,z]_T=[v_0,u]_T\cap [v_0,v]_T$ which we denote by
$\min(u,v)$.  Moreover, it follows in this case that 
$d_T(u,v)=d_T(u,z)+d_T(z,v)$.
\end{enumerate}
\end{exc}

\begin{prop}\label{P:2.2} The map $f: (V(T), d_T)\to \ell_1(E(T),w)$, $v\mapsto 1_{E([v_0,v]_T}$ is an isometry.

Here  $1_{E([v_0,v]_T)}: E(T)\to \R$ is seen as element of $\ell_1(E(T),w)$  defined by
$$1_{E([v_0,v]_T)}(e):=\begin{cases} 1 &\text{if $e\in E([v_0,v]_T)$,}\\
                                                                   0 &\text{else.}\end{cases}$$
\end{prop}
\begin{proof} Let $u,v\in V(T)$ and let $z=\min(v,u)$.
Since $[v,z]_T\cup[ z,u]_T $ is the (unique!) path from $u$ to $v$, it follows that
\begin{align*}
d_T(u,v)&= \sum_{e\in E([u,z])} w(e) +\sum_{e\in E([z,v])} w(e)\\
&= \sum_{E\in E(T)} |1_{E([u,z])}(e) - 1_{E([z,v])}(e)| w(e)\\
&=  \sum_{E\in E(T)} |1_{E([v_0,u])}(e) - 1_{E([v_0,v])}(e)| w(e)=\| f(u)-f(v)\|_1.
\end{align*}
\end{proof}
We use now the Extension property of $\cF(V(T),d_T)$ and denote the unique linear extension of $f$ by $F$, and recall that 
$\|F\|_{\cF(V(T),d_T)\to \ell_1(E(T),w)} =1$. By linearity of $F: \cF(M)\to\ell_1(E(T),w)$ we deduce that for any
$\mu=\mu^+-\mu^-$ and any $\nu\in \cM^+(\mu^+,\mu^-)$, that 
\begin{align*}
F(\mu)= F\Big( \sum_{x,y\in V(T)} \nu(x,y) (\delta_x-\delta_y)\Big)
           =\sum_{x,y\in V(T)} \nu(x,y) (f(x)-f(y))
           \end{align*}
 and thus 
\begin{align}\label{E:2.1} \|F(\mu)\|_1\le \|\mu\|_{\cF} = \|(\mu)\|_{\tc} &= \inf_{\nu\in \cM^+(\mu^+,\mu^-)}  \Big\|\sum_{x,y\in V(T)} \nu(x,y) (f(x)-f(y))\Big\|_1.
\end{align}
For our next step, we introduce another notation:
Let $e=(e^-,e^+)\in E_d(T)$, and  put 
$$V_e=\{ v\in V(T): e^{+} \preceq v\},$$
We note that $T_e=(V_e,E(T)\cap[V_e]^2) $ and its complement
$T^c_e=(V(T)\setminus V_e,E(T)\cap[V(T)\setminus V_e]^2)$ are subtrees of $T$.

\begin{prop}\label{P:2.3} For any $\mu\in  F(V(T), d_T)$ it follows that 
$$\|F(\mu)\|_1=\sum_{e\in E} w_e|\mu(V(T_e))|.$$
\end{prop}
\begin{rem} Since for $e\in E(T)$ the identicator  function $1_{T_e}$, as function on $V(T)$, is a Lipschitz function whose Lipschitz norm is $\frac1{w(e)}$ the term 
$\mu(T_e)$ is therefore well defined, even if $T$ is not a finite tree.
\end{rem} 

\begin{proof} W.lo.g. $\mu\in \tilde\cF(M)$.

By linearity
$$F(\mu)=\sum_{v\in V(T)} \mu(v) f(v)= \sum_{v\in V(T)}\mu(v) 1_{E([v_0,v])}$$
and thus 
\begin{align*}\|F(\mu)\|_1=\sum_{e\in E(T)} w(e) \Big|\sum_{v\in V(T)} \mu(v) \underbrace{1_{E([v_0,v])}(e)}_{=1\iff v\succeq e^+\iff v\in T_e}\Big|=\sum_{e\in E(T)} w(e) |\mu(T_e)|.
\end{align*}
\end{proof}
We are now able to provide an explicit formula for $\|\mu\|_{\tc}$, for $\mu\in \cF(V(T),d_T)$
\begin{thm}\label{T:2.4} For $\mu\in \cF(M)$
\begin{equation}\label{E:2.4.1a}\|\mu\|_{tc} =\sum_{e\in E} w_e|\mu(T_e))|.\end{equation}
\end{thm}

\begin{proof} It is enough to show \eqref{E:2.4.1a} for $\mu\in \tilde\cF(V(T),d_T)$. 

From  the inequality \eqref{E:2.1}  and  Proposition \ref{P:2.3} we deduce 
$$\sum_{e\in E} w_e|\mu(V(T_e))|=\|F(\mu)\|_1\le \|\mu\|_{\tc}  .$$
We need to show that there is a specific transportation plan $\nu\in\cM^+(\mu^+,\mu^-)$ 
for which
$$t(\nu)=\sum_{x,y\in V(T)} \nu(x,y) d_T(x,y)=  \sum_{e\in E} w_e|\mu(V(T_e))|.$$ 
We put 
$$S_{\mu} =\bigcup\{ [v_0,v]_T: v\in \supp(\mu)\} $$

Note that $T_\mu=(S_{\mu},E(T)\cap [S_\mu]^2)$ is a finite subtree of $T$,
we put  $n_\mu= |S_{\mu}|$  and we will verify our claim by induction for all values of $n_{\nu}$.

If $n_\nu=0$, it follows that $\mu=0$; thus, our claim is trivial.
Assume that $n_\mu=n+1$ and that our claim is true as long as $n_\mu\le n$.

We choose a leaf   $v$  of $T_\mu$. Note $v \not=v_0$ 
(otherwise $\mu$ would be non zero multiple of $\delta_{v_0}$) and let $u$ be its immediate predecessor.
Then we define $\mu'= \mu- \mu(v)\delta_v+ \mu(v)\delta_u$.
It follows that  
\begin{enumerate}
\item $S_{\mu'}\subset S_{\mu}\setminus \{v\}$, and thus  $n_{\mu'}<n_\mu$,
\item  $\mu(T_{(u,v)})=\mu(v)$ and $\mu'( T_{(u,v)}  )=0$,
\item  $\mu(T_e)=\mu'(T_e)$  for all $e\in E(T)\setminus\{(u,v)\}$.
\end{enumerate}
Using the induction hypothesis, let 
 $\nu'\in \cM^+({\mu'}^+,{\mu'}^-)$ be such that 
 $$\|\mu'\|_{\tc} = \sum_{x,y} \nu'(x,y) d(x,y)= \sum_{e\in E(T)}  w_e|\mu'(T_e)|=\sum_{e\in E(T), e\not=(u,v)}  w_e|\mu(T_e)|.$$

If $\mu(v)>0$ then  $\nu:= \nu'+\mu(v) (\delta_v-\delta_u)\in \cM^+(\mu^+,\mu^-)$ and if 
$\mu(v)<0$  then   $\nu:= \nu'+|\mu(v)|  (\delta_u-\delta_v) \in \cM^+(\mu^+,\mu^-)$, and in both cases
$$\sum_{x,y\in V(T)} \nu(x,y) d(x,y)= \sum_{x,y\in V(T)} \nu'(x,y) d(x,y)+ |\mu(v)| d(u,v)= \sum_{e\in T}  w(e) |\mu(T_e|.$$
\end{proof}

\begin{cor}\label{C:B.6.4} $\cF(V(T),{d_T})$ is isometrically isomorphic to $\ell_1(E(T),w)$ via the operator 
$$I: \cF(V(T),d_T)\to \ell_1(E(T),w),\quad  \sigma\mapsto \big(w_e\sigma(T_e): e\in E(T)\big).$$
\end{cor}
\begin{proof} For $\mu\in \cF(V(T),d_T)$  we have, by Theorem \ref{T:2.4}, $\|I(\mu)\|= \sum_{e\in E(T) }w_e|\mu(T_e)|=\|\mu\|_{\tc}$.
Since
$$I\Big(\frac{\delta_{e^+}-\delta_{e^-}}{d(e)}:e\in E(T)\Big)$$  is the unit vector basis of $\ell_1(E(T),w)$ $I$ is also surjective.
\end{proof}

\begin{rem} Alain Godard \cite{Godard2020} proved Corollary  \ref{C:B.6.4} for more general trees ($\R$-trees), we followed a simplified version of his arguments.
 The most general version of Corollary  \ref{C:B.6.4} can be found in \cite{Dalet_Kaufmann_Prochazka2016}.
In \cite{Godard2020}, Godard also proved a converse, namely that every metric space $(M,d)$, for which $\cF(M,d)$ isometrically to  $L_1$, $ \ell_1$ or $\ell_1^N$ , is isometrically equivalent to an  $\R$-tree.
We will prove that fact for a finite metric space (in which case the proof is much simpler).
\end{rem}

\begin{prop}\label{P:B.6.6} Assume $(M,d)$ is a finite metric space  and $\cF(M)$ is isometrically equivalent to $\ell_1^{n}$ for some $n\in\N$.

Then $(M,d)$ is isometrically equivalent to a geodesic tree, meaning there is a $E\subset [M]^2$ so that $T=(M,E)$ is a tree and $d$ is a geodesic metric for $T$.
\end{prop} 

We will use the following fact:
\begin{prop}\label{P:B.6.7} A Banach space $E$  of dimension $n\in\N$ is isometric to $\ell_1^n$ if and only 
if there are $x_1,x_2,\ldots,x_n \in S_E$ for which 
$\{\pm x_j: j=1,2,\ldots n\}$ are the extreme points of $B_E$.
\end{prop} 

\begin{proof}[Proof  of Proposition \ref{P:B.6.6}]
Define
$$E=\big\{\{x,y\}\in [M]^2: \text{ there does not exist }z\in M\setminus\{x,y\} \text{ with } d(x,y)=d(x,z)+d(z,y)\big\}.$$
Then $G=(M,E)$ is a connected graph and $d$ is a geodesic metric with respect to $G$. 
Thus, by Theorem \ref{T:1.11} the extreme points  of $B_{\cF(M)}$ are 
$$ \Big\{\pm \frac{ \delta_{u}-\delta_{v}}{d(u,v)}: \{u,v\}\in E\Big\}.$$

Assume that $\cF(M) $ is isometrically equivalent to $\ell_1^n$. Thus, 
$\dim(\cF(M))=n$, which implies that  $|M|=n+1$, and  
$B_{\cF(M)}$ has $2n$ extreme points, which means that 
the cardinality of $E$ must be $n$. So $(M,E)$ is a connected graph with $n+1$ vertices and $n$ edges. This implies by Proposition \ref{P:2.1} 
that $(M,E)$ is a tree.\end{proof}

\section{Stochastic Embeddings of metric spaces into trees}\label{S:3}

\subsection{Definition of Stochastic Embeddings, Examples}

  \begin{defn}\label{D:3.1} Let $\cM$ be a class of metric spaces, and let  $(M,d_M)$  be a  metric space.
  A family $(f_i)_{i=1}^n$ of maps $f_i:M\to M_i$, with $(M_i,d_i)\in\cM$, together with numbers $\P=(p_i)_{i=1}^\infty\subset(0,1]$, with $\sum_{i=}^\infty p_i=1$    is called a 
  {\em $D$-stochastic embedding of $M$ into elements of the class  $\cM$} if  for all $x,y\in M$ and $i=1,2\ldots,$
  \begin{align}
  \label{E:3.1.1}&d_M(x,y)\le  d_i(f_i(x),f_i(y)) \text{ (expansiveness)},\\
   \label{E:3.1.2}&\E_\P\big(d_i(f_i(x),f_i(y)\big)=\sum_{i=1}^\infty p_i d_i\big(f_i(x),f_i(y) \big)\le D d_M(x,y).
  \end{align} 
  In that case we say that $(M,d_M)$  is {\em $D$-stochastically  embeds  into   $\cM$}. If moreover 
  the maps $f_i: M\to M_i$ are bijections we 
  say that  $(f_i)_{i=1}^n$  together with  $(p_i)_{i=1}^n$ is  a {\em bijective $D$-stochastic embedding of $M$ into elements of the class  $\cM$}.  \end{defn}  
\begin{rem} Of course if  $(f_i)_{i=1}^\infty$,  $f_i:M\to M_i$, $i=1,2,3,\ldots $  together with numbers   $(p_i)_{i=1}^\infty$  is a
  bijective $D$-stochastic embedding of $M$ into elements of the class  $\cM$, we can assume that as sets $M_i=M$, and that $d_i$ is a metric on $M$.
\end{rem}

We will mainly be interested in how a finite metric graph can be stochastically embedded into trees.

\begin{ex} Let $C_n=(V(C_n), E(C_n))$ be a cycle of length $n$. We can write $V(C_n)$ and $E(C_n)$ as 
$V(C_n)=\Z/nZ$ and $E(C_n)=\big\{\{j-1,j\}: j=1,2,\ldots n\big\}$
and let $d$ be the geodesic metric generated by the constant weight function $1$.

Consider for $j_0=1,2,\ldots n$, the path  $P_{j_0}$ defined by
$V(P_j)=V(C_n)=\{0,1,2,\ldots, n-1\}$
and $$E(P_{j_0})= E(C_n)\setminus \big\{\{j_0-1,j_0\}\big\}.$$ 
We consider  on $P_{j_0}$ the (usual {\em path distance} ) generated
by the weight function $w(e)= 1$, for $e\in P_{j_0}(e)$, and denote it by $d_{j_0}$

It follows  that
$$d_{j_0}(j_0-1,j_0)= n-1$$
and thus it follows for the identity $Id: (V(C_n),d_{C_n})\to (V(C_n), d_{j_0})$
that $\dist(Id)=n-1$. 

Let $\cT$ be the set of all metric trees. It can be shown that the $\cT$-distortion of a cycle $C_n$ of length  
$n$ is of the order $n$, ie\ $c_{T}(C_n)\ge c\!\cdot\! n$.
In other words, there are no embeddings of cycles into trees with sublinear (with respect to $n$) distortion.

Nevertheless, the  distortion of stochastic embedding of $C_n$  into trees is not larger than $2$:

$I=\{1,2,\ldots n\}$, $p_i=\frac1n$, and let for $i\in I$, $P_i$ be  the above introduced path.
Then 
$d_{C_n}(u,v)\le d_i(u,v)$, for $u,v\in V(C_n)$, and for $e=\{j-1,j\}\in E(C_n)$ it follows that
$$\sum_{i=1}^n p_i d_i(j-1,j)=\frac1n( 2(n-1))<2.$$

\end{ex}

\subsection{Stochastic Embeddings into Trees: the Theorem of  Fakcharoenphol, Rao, and Talwar}

\begin{thm}\label{T:3.2} {\rm \cite{FakcharoenpholRaoTalwar2004}} Let $M$ be a metric space with $n\in\N$ elements. Then there is a $O(\log n)$ stochastic embedding of $M$ into the class of weighted trees. 
\end{thm}

A complete proof of Theorem \ref{T:3.2} is given in the Appendix. Here, we only want to define the stochastic embedding.

Let $M=\{x_1,x_2,\ldots ,x_n\}$. After rescaling, we can assume that $d(x,y)>1$ for all $x\not=y$ in $M$.

We introduce the following notation.
\begin{enumerate} 
\item $B_r(x)=\{ z\in M: d(z,x)\le r\}$  for  $x\in X$ and $r>0$. 
\item For $A\subset M$, $\diam(A)=\max_{x,y\in A} d(x,y)$. \\
     We choose $k\in \N$ so that $2^{k-1}<\diam(M)\le 2^k$.
      
\item $\cP_M$ denotes the sets of all partitions of $M$. Let  $P=\{A_1,A_2, \ldots A_l\}\in\cP_M$.\\
          $P$ is called an $r$-Partition if $\diam (A_j)\le r$, $j=1,2\ldots l$\\
          for $B\subset M$, let  $P|_B=(A_1\cap B,A_2\cap B, \ldots,A_l\cap B)$
 \item For two partitions $P=\{A_1,A_2,\ldots A_m\}$ and $Q=\{B_1,B_2,\ldots,B_n  \}$ 
          we say $Q$ subdivides  $P$ and write $Q\succeq P$, if 
          for each $i=1,2,\ldots,m $ there is a $j=1,2\ldots ,n$ with $B_j\subset A_i$
  \item     Let $\Omega$ be the set of all sequences
             $(P^{(j)}  )_{j=0}^k$ so that:
             $P^{(j)} $ is a $2^{k-j}$-partition of $M$, with $P^{(0)}=(M)$ and $P^{(j)}\succeq P^{(j-1)}$ \\
             Note: $P^{(k)}$ are singletons.
  \item Every  $(P^{(j)}  )_{j=0}^k\in \Omega$ defines a tree $T=(V(T),E(T))$ as follows:\\
           $V   (T)=\{ (j,A): j=0,1,2,\ldots k, \text{ and } A\in P^{(j)}$\}, and\\ 
           $E(T)=\big\{\{(j,B), (j-1, A)\}  :  j=1,2,\ldots k\text{ and }A\subset   B\}\big\}$.\\
          The weight function is defined by  $w(\{(j,B), (j-1, A)\})=2^{k-j}$ if $\{(j,B), (j-1, A)\}\in E(T)$ 
\end{enumerate}
For each such  tree $T$, we can define a map from $M$ to the leaves of $T$ by assigning
each $x$ the element $(k, \{x\})\in V(T)$.

 We now have to define a  Probability on $\Omega$.
 
 We first define for each $R$ a probability $\P^R$ on $R$-bounded partitions:
 Let $\Pi_n$ be the set of all permutations on $\{1,2,\ldots n\}$
 
 We consider on $\Pi\times [R/4, R/2]$ the product of the uniform distribution on $\Pi$ and the   uniform distribution on $[R/4, R/2]$.
 each pair $(\pi, r)\in   \Pi\times [R/4, R/2]$ defines the following $r$-partition $(C_j{(\pi, r)})_{j=1}^l$:\\
 $\tilde C_1{(\pi, r)}=B_r(x_{\pi(1)})$, and assuming   $(\tilde C_1{(\pi, r)},\tilde C_2{(\pi, r)}), \ldots ,\tilde C_{j-1}{(\pi, r)}$ are defined 
 we put $\tilde C_j{(\pi, r)}=B_r(x_{\pi(j)})\setminus \bigcup_{i=1}^{j-1} \tilde C_i{(\pi, r)}$.
Then let $\{C_j{(\pi, r)}: j\keq1,2\ldots l\}$ be the non empty elements of $(\tilde C_j{(\pi, r)})_{j=1}^n$
and let $\P^R$ be the image distribution of that mapping $(\pi, r)\mapsto C(\pi,r)$.
For $B\subset M$ let $\P^{(R,B)}$ be the image distribution of the mapping $(\pi, r)\mapsto C(\pi,r)|_B$
 
Let $\P$ be the probability on $\Omega$ uniquely defined by 
 \begin{align} \P\big((P^{(j)}  )_{j=0}^k: P^{(0)}&=(0,\{M\})\big)=1,
 \intertext{for $j=1,2\ldots n$, $B\subset M$,  $\diam(B) \le  2^{k-(i-1)}$, and any $\cA\subset\ \cP_M$}
\P( P^{(i)}|_B \in \cA|(i-1,B)\in P^{(i-1)})&= \P^{(2^{k-i}, B)}(\cA).
 \end{align} 
\subsection{Bijective  embeddings onto trees: The Restriction Theorem by Gupta}

Our goal is to show that for some universal constant $c$, every metric space with $n$ elements can be bijectively $c\log(n)$-stochastic embedded into trees. This will follow from Theorem \ref{T:3.2}  and the following result by Gupta.
\begin{thm}\label{T:3.3} {\rm \cite[Theorem 1.1]{Gupta2001}}  Let $T=(V,E,W)$ be a weighted tree and $V'\subset V$. Then there is $E\subset [V']^2$ and $W': E(\G')\to [0,\infty)$
so that $T'=(V',E(\G'),W')$ is a tree 
\begin{equation}\label{E:13.4.1a} \frac14\le \frac{d_{T'}(x,y)}{d_{T}(x,y)}\le 2,\text{ for $x,y\in V'$}.\end{equation}
\end{thm}
A proof of Theorem \ref{T:3.3} is given in the appendix.
\begin{cor}\label{C:13.7} If  a finite metric space $(M,d)$ $D$ stochastically embeds into geodesic trees,
it bijectively $8$-stochastically embeds into geodesic trees.

In particular,  by \ref{T:3.2} is a constant $c>0$ so that  every metric space $(M,d)$  with $n$ elements 
bijectively $c\log(n)$-stochastically embeds into a weighted tree.
\end{cor}

\begin{exc} The estimate  in Corollary \ref{C:13.7} is optimal in the following sense:

Let $\G_n=(\Z/n\Z)^2$ is the $n$-t discrete torus, then there is constant $c>0$ so that 
for all $n\in \N$ the following holds:

If $T$ is a tree and $f:V(\G_n)\to V(T)$ is in an injective map so that 
$$1\le d_T(f(x),f(y))\text{ for all $x,y\in \G_n$, with $\{x,y\}\in E(\G)$}$$
then 
$$\frac1{|E(\G)|} \sum_{e=\{x,y\}\in E(\G)} d_T (f(x),f(y))\ge c\log(n)$$

The same is true for the family of diamond graphs $(D_n)_{n\in\N}$ (see \cite{GuptaNewmanRabinovichSinclair2004}*{Theorem 5.6}).
\end{exc}

\begin{rem} Assume that $(G, d_G)$ is a finite geodesic graph,
 and that the geodesic trees $(T_i,d_i)$, expansive maps $f_i: V(T_i)\to V(G)$, $i=1,2, \ldots n$, together with the probability
 $\P=(p_i)_{i=1}^n$ form a bijective $D$-stochastic embedding of $V(G)$. Then we can assume that actually 
 $V(T)=V(G)$,  and thus  $E(T_i)\subset [V(G)]^2$, and that $f_i$ is the identity.
 
 For $i=1,2,\ldots n$ and $e=\{x,y\}\in E(T_i)$ (which does not need to be in E(G)) 
 $$w'_i(e)= d_G(x,y)=\min\big\{ \length_{d_G}(P): \text{$P$ is a path in $G$ from $x$ to $y$}\big\}\le d_i(e).$$
 and let $d'_i$ be the geodesic metric on $V(T_i)=V(G)$ generated by $w'_i$.
 
 We note that $d_i'(x,y)\le d_i(x,y)$, for $x,y\in V(G)$, but that the identity 
 $(V(G),d_G)\to (V(G),d_i')$ is still expansive.
 It follows therefore that if we replace $d_i$ by $d'_i$ we do not increase the stochastic distortion, may only reduce it.
 \end{rem}

\subsection{Application: Embedding Transportation Cost Spaces  into $\ell_1$}\label{SS:3.4}

\begin{thm}\label{T:3.4.1}{\rm \cite{BMSZ2022,Mathey-PrevotValette2023}} Let $(M,d)$ be a countable  metric space which $D$-stochastically embeds into geodesic   trees.
Then, there  is an isomorphic embedding
$$\Phi: \cF(M)\to \ell_1^\infty$$ with
$$\|\mu\|_\cF\le \|\Phi(\mu)\|_1\le D\|\mu\|_\cF \text{ for all $\mu\in \cF(\cM)$}.$$
   \end{thm} 
We will use the following easy fact:
Let $\P$ be a probability on a finite or countable infinite set $I$, with $\P(i)>0$, for $i\in I$.
Put
$$L_1(\P,\ell_1)=\{ f:I\to \ell_1\}$$
and for $f\in L_1(\P,\ell_1)$, we write  $f(i)=\sum_{j=1}^\infty e_j f(i,j)$,  where $(e_j)$ denotes the unit vector basis of $\ell_1$, and put
$$\|f\|_{L_1(\P,\ell_1)}= \int_I \|f\|_1 d\P= \sum_{i\in I} \|f(i)\|_1 \P(i)= \sum_{i\in I} \sum_{j=1}^\infty |f(i,j)| $$
Then 
$$L_1(\P,\ell_1)\to \ell_1^{I\times \N},\quad f\mapsto \Big( \frac{f(i,j)}{\P(i)}: i\in I,j\in \N\Big).$$
is an onto isometry.

\begin{proof}[Proof of Theorem \ref{T:3.4.1}]
For $i\in I$ let  $(T_i,d_i)$   let  $d_i$ be  a geodesic  metric on $T_i$, $\P=(p_i)_{i\in I}$ a (strictly positive) probability on $I$,
 and let $\phi_i: M\to V(T_i)$ so that 
 \begin{align*}
 &d_i(\phi_i(x), \phi(y))\ge d(x,y)\text{ for $i\kin I$ and $x,y\kin M$}, \text{ and}
  &\sum_{i\in I} p_i d_i(\phi_i(x), \phi_i(y))\le D d(x,y)\text{ for $x,y\kin M$.}
  \end{align*}
  We can assume that $T_i$ are countable trees. By  Corollary \ref{C:B.6.4}
  $\cF(V(T_i), d_i)$ is isometrically isomorphic to $\ell_1^{(E(T_i))}$, for $i\in I$,
  and thus, there are 
   isometric embeddings 
   $E_i :\cF(V(T_i,) d_i)\to \ell_1$, for $i\in I$.
   
   We define 
   $$\Psi_i:\tilde\cF(M,d)\to \cF(V(T_i), d_i), \quad \sum_{j=1}^n r_i(\delta_{x_i}-\delta_{y_i})\mapsto  \sum_{j=1}^n r_i(\delta_{\phi_i(x_i)}-\delta_{\phi(y_i)}).$$
 Note that if  $\mu\in \tilde\cF(M,d)$ is represented by
 $\mu=\sum_{j=1}^n r_i(\delta_{x_i}-\delta_{y_i})$ then $\Psi_i(\mu)$ is represented
 by $\Psi_i(\mu)=\sum_{j=1}^n r_j(\delta_{\phi_i(x_j)}-\delta_{\phi_i(y_j)})$, and thus
 \begin{align*}
 \|\Psi_i(\mu)\|_{tc} &= \inf\Big\{ \sum_{j=1}^n r_jd_i(u_j,v_j) : \Psi_i(\mu)=\sum_{j=1}^n r_j(\delta_{u_j}-\delta_{v_j}), \,(r_j)_{j=1}^n\subset \R^+\Big\}\\
                                &= \inf\Big\{\sum_{j=1}^n r_j d_i(\phi_i(x_j),\phi_i(y_j)):\mu=\sum_{j=1}^n r_i(\delta_{x_j}-\delta_{y_j}),\, (r_j)_{j=1}^n\subset \R^+\Big\} 
                                \ge \|\mu\|_{tc}.\\
   \end{align*}
   For the second equality note that ``$\le$'' is clear and that ``$\ge$'' follows from the fact (see Proposition \ref{P:1.4}) that among the optimal representations of
   $\Psi_i(\mu)$, there always must be one whose support is the support of $\Psi_i(\mu)$, which is a subset of $\phi_i(M)$.
   
      We define 
   $$\Psi: \tilde\cF(M,d)\to L_1(\P, \ell_1)\equiv \ell_1, \quad \mu\mapsto \Psi(\mu): I\ni i \to E_i\circ \Psi_i(\mu)\in \ell_1.$$
  
  Then it follows for $\mu\in \tilde\cF(M,d)$, that 
  $$\|\Psi(\mu)\|_{L_1(\P,\ell_1)}=\sum_{i\in I}\pi \|\Psi_i(\mu)\|_{tc}\ge \|\mu\|_{tc}$$
  and on the other hand if $\mu=\sum_{j=1}^n r_j (\delta_{x_j}-\delta_{y_j})$ is an optimal representation of $\mu$
  then 
    \begin{align*}
    \|\Psi(\mu)\|_{L_1(\P,\ell_1)}&=\sum_{i\in I}p_i \|\Psi_i(\mu)\|_{tc}\le \sum_{i\in I}p_i\sum_{j=1}^n r_j d(\phi_i(x_j),\phi_i(y_j))\\
    &\le D\sum_{i\in I}p_i\sum_{j=1}^n r_j d(x_j,y_j))=D\|\mu\|_{tc}.\end{align*}

  \end{proof}
%
%
%

\subsection{Application: Extension of Integral Operators from a Conservative Field to the whole Vector Field,
and the  Embedding  of Transportation Cost Spaces into $L_1$  complementably}\label{SS:3.5}\phantom{aa}

We recall some notation from {\em discrete Calculus}.
We are given a finite graph $G=(V(G),E(G))$ with a geodesic metric $d_G$.
We put $\bar E(G)=\big\{ (x,y), (y,x): \{x,y\} \in E(G)\big\}$.

A map $f: \bar E(G)\to \R$ is called a {\em vector field on $G$} if $f(x,y)=-f(y,x)$ for all $\{x,y\}\in E(G)$,
and we put  $$\|f\|_\infty =\sup_{e=(x,y)\in \bar E(G)} |f(x,y)|.$$ 
We denote the space of vector fields on $G$ by  $\VF(G)$. Together with this norm $\VF(G)\equiv\ell_\infty(E(G))$

If $W=(x_j)_{j=0}^n$ is a walk in $G$ and $f\in \VF(G)$, we call the {\em  integral of $f$ along $W$} to be
$$\int_W f(e) d_G(e)=\sum_{j=1}^n f(x_{j-1}, x_j) d_G(x_{j-1}, x_j).$$
A vector field $f$ on $G$ is called {\em conservative } if the integral along any cycle $C$ vanishes,
\ie\ if $C=(x_j)_{j=1}^n$ is a walk with $x_n=x_0$, then
$$\int_C f(e) d_G(e)=\sum_{j=1}^n f(x_{j-1}, x_j) d_G(x_{j-1}, x_j)=0.$$
Equivalently, if for any $x,y\in V(G)$ and any two paths $P$ and $Q$ from $x$ to $y$
it follows that 
$$\int_P f(e) d_G(e)=\int_Q f(e) d_G(e).$$
We denote the space of conservative vector fields on $G$ by $\CVF(G)$.

\begin{rem} If $T$ is a tree, then $\VF(T)=\CVF(T)$.
\end{rem}

For a map $F: V(G)\to \R$ we define the {\em gradient of $F$} by
$$\nabla F=\nabla_dF: \bar E\to \R, \quad (x,y)\mapsto \frac{f(y)-f(x)}{d_G(x,y)}.$$

\begin{prop} For $F:V(G)\to \R$
$$\|F\|_{\Lip} =\|\nabla F\|_\infty.$$
\end{prop}
\begin{proof} We need to observe that
$$\sup_{x,y\in V(G), x\not=y} \frac{|f(y)-f(x)|}{d_G(x,y)}=\sup_{\{x,y\}\in E(G)} \frac{|f(y)-f(x)|}{d_G(x,y)}.$$
Indeed, let $x,y\in V(G)$, $x\not=y$ and $P=(x_j)_{j=0}^n$ path of shortest metric length from $x$ to $y$
then 
\begin{align*}
\frac{|f(y)-f(x)|}{d_G(x,y)}&\le  \frac{\Big|\sum_{j=1}^nf(x_j)-f(x_{j-1})\Big|}{\sum_{j=1}^n d_G(x_{j-1}, x_j)}\\
&\le  \frac{\sum_{j=1}^n|f(x_j)-f(x_{j-1})|}{\sum_{j=1}^n d_G(x_{j-1}, x_j)}\le \max_{j=1,2\ldots} \frac{|f(x_j)-f(x_{j-1})|}{d_G(x_{j-1}, x_j)},
\end{align*}
where the last inequality follows from iteratively 
applying the inequality
$$\frac{a+b}{c+d} \le \max\Big(\frac{a}{c},\frac{b}{d}\Big).$$
\end{proof}

\begin{prop}\label{P:3.4.1}
Fix a point $0\in V(G)$.
Then for every $F\in \Lip_0(V(G))$, it follows that
$\nabla F\in \CVF(G)$ and 
$$\| F\|_\Lip= \|\nabla F\|_\infty.$$
Moreover
$$\nabla: \Lip_0(V(G))\to \CVF(G)$$
is a surjective isometry whose inverse is 
defined by the {\em Integral operator}
$$I(f)(x)=\int_P f(e) d(e), \text{ for $x\in V(G)$.}$$
where $P$ is any path from $0$ to $x$, for $x\in V(G)$.
\end{prop}

\noindent{\bf Question:} How well is $\CVF(G)$ complemented in $\VF(G)$
Equivalently, what is the smallest constant  $C\ge 1$, so that 
the operator 
$$I:\CVF(G)\to \Lip_0(V(G),d)$$
can be extended to an  operator
$$ \tilde I: \VF(G)\to \Lip_0(V(G),d).$$

\begin{thm}\label{T:3.4}  Assume that $(V(G),d_G)$ is finite $D$-stochastically embeds into geodesic trees. Then
the integral operator $I:\CVF(G)\to \Lip_0(V(G),d_G)$ can be extended to an operator 
$\tilde I: \VF(G)\to \Lip_0(V(G),d)$,
with $\|\tilde I\|\le 8D$.
\end{thm}
\begin{proof}  Theorem \ref{T:3.3} and the assumption imply that  $(V(G),d_G)$ bijectively  $8D$-stochastically embeds into geodesic trees.
We can find $n\in\N$, geodesic trees $(T_i, d_i)$, with $V(T_i)=V(G)$, for $i=1,2\ldots, n$, and a probability $\P=( p_j)_{j=1}^n$ ,
so that 
\begin{align}
&\label{E:3.4.1.1}d_G(x,y)\le d_i(x,y) \text{ for all $i=1,2\ldots, n$, and $x,y\in V(G)$, and }\\
&\label{E:3.4.1.2}\sum_{j=1}^n p_i d_i(x,y)\le 8Dd(x,y),\text{ for all  $x,y\in V(G)$.}
\end{align}
By the Remark after Theorem 3.3. we can assume that for each $i=1,2,\ldots, n$ and any
$e=\{x,y\}\in E(T_i)$, it follows that $d_i(x,y)=d_G(x,y)$ and we will choose 
a path $P_e$ in $G$ from $x$ to $y$, so that
$$d_i(x,y)= \length_{d_G}(P_e)= d_G(x,y).$$

For $f\in VF(G)$, and $i=1,2\ldots, n$ we define $f^{(i)}\in \VF(T_i)$ (=$\CVF(T_i)$) by
$$f^{(i)}(e)=\frac1{d_i(x,y)} \int_{P_e} f(e') d_G(e') = \frac1{d_G(x,y)} \int_{P_e} f(e') d_G(e')   \text{ for $e\in E(T_i)$}.$$
From \eqref{E:3.4.1.1} we deduce that 
\begin{equation}
\label{E:3.4.1.3} |f^{(i)}(e)|\le \|f\|_\infty \frac{d_G(x,y)}{d_i(x,y)}\le \|f\|_\infty \text{ for all $i=1,2,\ldots,n$ and $e\in E(T_i)$.}
\end{equation}
Denote the (unique) path from $x$ to $y$ in $T_i$ by $[x,y]_i$.\
We define $\tilde I(f)\in \Lip_0(V(G),d_G)$ by 
\begin{align*}\tilde I(f)(x)&:=\sum_{i=1}^n p_i \int_{[0,x]_i} f^{(i)}(e)d_i(e)\\
&=\sum_{i=1}^n p_i \sum_{e\in E([0,x]_i)} f^{(i)}(e) d_i(e)= \sum_{i=1}^n p_i \sum_{e\in E([0,x]_i)}\int_{P_e} f(e') d_G(e').\end{align*}
We note:
\begin{itemize}
\item
If $f$ is a conservative field, then $\tilde I(f)=I(f)$. Indeed,
  denote for $i=1,2,\ldots, n$  and $x\in V(G)$ the walk in $G$ from $0$ to $x$,  obtained by concatenating the paths $P_e$,  $e\in E( [0,x]_i)$,  by $W_i$
 Then
$$\tilde I(f)(x)=\sum_{i\in I}p_i\sum_{e\in [0,x]_i}\int_{p_e} f d_G(e)  = \sum_{i\in I}p_i \int_{W_i} f(e) d_G(e) =I(f)(x).$$

\item For  $e=\{x,y\}\in E$, it follows from \eqref{E:3.4.1.3} that 
\begin{align*}
|\tilde I(f) (y)-\tilde I(f)(x)|&=\Bigg| \sum_{i\in I} p_i \int_{[x,y]_i} f^{(i)}(e) d_i(e)\Bigg| \\
&\le  \sum_{i\in I} p_i \Bigg|\int_{[x,y]_i} f^{(i)}(e) d_i(e)\Bigg|\\
               &\le \|f\|_\infty  \sum_{i\in I} p_i   \sum_{e\in [x,y]_i} d_i(e)\\
               &= \|f\|_\infty\sum_{i\in I} p_i d_i(x,y)\le 8Dd_G(x,y)\|f\|_\infty,\end{align*}
               
\end{itemize}
and, thus, $\|\tilde I(f)\|_\infty \le 8D \|f\|_\infty$.
\end{proof}
\begin{cor} If $(V(G),d_G)$ is a finite graph which  $D$-stochastically embeds into geodesic trees, then
$\cF(V(G),d_G)$ is $8D$-complemented in  $\ell_1(E(G))$.
\end{cor} 
\begin{proof} By Theorem  \ref{T:3.2} 
$\Lip_0(V(G),d_G)\equiv \CVF(G)$ is  $8D$ complemented in $\VF(G)\equiv \ell_\infty(E(G))$. Passing to the dual 
we obtain that $\cF(V(G),d_G)$ is $8D$ complemented in $\ell_\infty(E(G))$.
\end{proof} 

Now assume that $G=(V(G),E(G))$ is a countable graph with geodesic metric $d_G$, which is $D$ stochastically embeddable into trees. 
We cannot use Gupta's result Theorem \ref{T:3.3}.
Therefore, we must also assume that $(G,d_G)$ is bijectively  $D$ stochastically embeddable into trees. 
So let $I\subset \N$, $\P=(p_i)_{i\in I} \subset (0,1]$ with $\sum_{i\in I} p_i=1$, and for $i\in I$ let $T_i=(V(T_i),E(T_i))$ be a tree with $V(T_i)=V(T)$ with a geodesic metric 
$d_i$,   with $d_i(e)=d_G(e)$, if $e\in E(T_i)$ so that:
\begin{align}\label{E:3.4.1} &d_i\big(x,y)\big)\ge d_G(x,y),\text{ for $i\in I$, and $x,y\in V(G)$},\\
\label{E:3.4.2}  &\E_P\big( d_i\big(f_i(x),f_i(y)\big)\big)=\sum_{i\in I} p_id_i\big(f_i(x),f_i(y)\big)\le Dd_G(x,y),\text{ for $x,y\in V(G)$.}
\end{align}
Choose for each $i\in I$ a root $v_i\in V(T_i)$ and define the Tree $T$, by {\em gluing the trees $T_i$ together at $v_i$}, \ie\ 
put $V(T)=\bigcup_{i\in I} \{i\}\times V(T_i)$, and identify $(i,0)$ with $(j,0)$ for $i,j\in I$, and denote this point by $0$ (the root of $T$)
and $E(T)=\bigcup_{i\in I}\big\{\{(i,x), (i,y)\}: \{x,y\}\in E(T_i)\big\} $.
We put  $w_e=d_i(x,y)$, for $e=\{(i,x), (i,y)\}\in E( T)$ (and thus $\{x,y\}\in E(T_i)$). This   defines  a weight function on $E(T)$, which generates a geodesic metric on $d_T$, which has the property 
that $d_T(e)=d_i(x,y)$, for  $e=\{(i,x), (i,y)\}\in T$. 

We direct the edges of $E(T)$ by choosing the orientation $((i,x), (i,y))$ for $\{(i,x), (i,y)\}\in T$
if $d(0,(i,x))<d(0,(i,y))$.

We now consider the maps $f_i:V(G)\to V(T)$, $x\mapsto (i,x)$  which satisfy for  $x,y\in V(G)$
\begin{align}\label{E:3.4.3} d_T\big(f_i(x),f_i(y)\big)\ge d_G(x,y) \text{ and }
 \E_P\big( d_T\big(f_i(x),f_i(y)\big)\big)\le Dd_G(x,y).
\end{align}
Again, as in Theorem \ref{T:3.4}, we can choose for each $e\in E_d(T)$, say $e=\big((i,x),(i,y)\big)$, with $i\in I$, a path in $G$ from $x$ to $y$, which we denote by $P_e$,
with $\length(P_e)= d_T(e)=d_i(e)$

For $f\in VF(G)$ and $i\in I$, and $e=\{(i,x),(i,y)\}\in E(T)$   we define
$$f^{(i)}(e)=\frac1{d_T(e)}\int_{P_e} f(e')\, d_G(e')=\frac1{d_G(e)}\int_{P_e} f(e') \,d_G(e')$$
From \eqref{E:3.4.3} we deduce that 
\begin{equation}
\label{E:3.4.5} |f^{(i)}(e)|\le \|f\|_\infty \frac{d_G(x,y)}{d_T(e)}\le \|f\|_\infty \text{ for all $i\in I$ and $e=((i,x),(i,y))\in E_d(T)$.}
\end{equation}
For $x,y\in V(G)$ and $i\in I$ denote the  (unique) path from $(i,x)$ to $(i,y)$ in $T$ by $[x,y]_i$.\
We define $\tilde I(f)\in \Lip_0(V(G),d_G)$ by 
\begin{align*}\tilde I(f)(x)&:=\sum_{i\in I} p_i \int_{[0,x]_i} f^{(i)}(e)d_T(e)\\
&=\sum_{i\in I} p_i \sum_{e\in E([0,x]_i)} f^{(i)}(e) d_i(e)= \sum_{i\in I} p_i \sum_{e\in E([0,x]_i)}\int_{P_e} f(e') d_G(e').\end{align*}
We note:
\begin{itemize}
\item
If $f$ is a conservative field, then $\tilde I(f)=I(f)$. Indeed,
  denote for $i\in I$  and $x\in V(G)$ the walk in $G$ from $0$ to $x$,  obtained by concatenating the paths $P_e$,  $e\in E( [0,x]_i)$,  by $W_i$.
 Then
$$\tilde I(f)(x)=\sum_{i\in I}p_i\sum_{e\in [0,x]_i}\int_{p_e} f d_G(e)  = \sum_{i\in I}p_i \int_{W_i} f(e) d_G(e) =I(f)(x).$$

\item For  $e=\{x,y\}\in E$, it follows from \eqref{E:3.4.5} that 
\begin{align*}
|\tilde I(f) (y)-\tilde I(f)(x)|&=\Bigg| \sum_{i\in I} p_i \int_{[x,y]_i} f^{(i)}(e) d_i(e)\Bigg| \\
&\le  \sum_{i\in I} p_i \Bigg|\int_{[x,y]_i} f^{(i)}(e) d_i(e)\Bigg|\\
               &\le \|f\|_\infty  \sum_{i\in I} p_i   \sum_{e\in [x,y]_i} d_i(e)\\
               &= \|f\|_\infty\sum_{i\in I} p_i d_i(x,y)\le 8Dd_G(x,y)\|f\|_\infty,\end{align*}
               
\end{itemize}
and, thus, $\|\tilde I(f)\|_\infty \le 8D \|f\|_\infty$
\noindent 

\section{Lower estimates for embeddings of $\cF(M)$ into $L_1$}\label{S:4}
In this last section, we want to formulate a criterion on geodesic graphs $(G,d)$
which implies that the distortion of embeddings of $\cF(V(G),d)$ has to satisfy lower estimates.

This will lead to sequences of geodesic graphs $(G_n, d_n)$, for which
$$c_{L_1}(\cF(V(G_n), d_n)\ge C\sqrt{\log(V(G_n))}.$$
 Among these sequences are, for example the sequence of discrete tori $\big((\Z/n\Z)^2: n\in \N\big)$ \cite{NaorSchechtman2007} and the  sequence of diamond graphs
 $(D_n:n\in \N)$,  \cite{BaudierGartlandSchlumprecht2023}. The idea of the proof goes back to a result of Kislyakov from 1975 \cite{Kislyakov75}.
 This result of Kislyakov implies, for example that $\cF(\R^2)$ is not isomorphic to $\cF(\R)$ (which is isometrically isomorphic to $L_1(\R)$).

  Throughout this section, $(G,d)$ is a geodesic finite graph and $\nu$  a probability on $E(G)$, whose support is all of $E(G)$. We define the probability  on $V(G)$, {\em induced by $\nu$ } as follows
 $$\mu(v)=\mu_\nu(v)=\frac12\sum_{e\in E(G), v\in e} \nu(e)  \text{$v\in V(G)$.}$$
 
 Note that 
 $$\sum_{v\in V(G)} \mu(v)=\sum_{v\in V(G)}\frac12\sum_{e\in E(G), v\in e} \nu(e) =\frac12 \sum_{e\in E(G)} \sum_{ v\in e} \nu(e)= 1,$$
 which shows that $\mu$ is indeed a probability on $\mu$

 \subsection{Isoperimetric dimension and the Sobolev inequality}\label{S:4.1}
 Let $(G,d)$ be a geodesic finite graph and let $\nu$ be a probability on $E(G)$. We define the probability  on $V(G)$, {\em induced by $\nu$ } as follows
 $$\mu(v)=\mu_\nu(v)=\frac12\sum_{e\in E(G), v\in e} \nu(e)  \text{$v\in V(G)$.}$$
 
 Note that 
 $$\sum_{v\in V(G)} \mu(v)=\sum_{v\in V(G)}\frac12\sum_{e\in E(G), v\in e} \nu(e) =\frac12 \sum_{e\in E(G)} \sum_{ v\in e} \nu(e)= 1.$$

  For $A\subset V(G)$ we define the {\em boundary }of $A$ to be
$$\partial_G A := \{ \{x,y\}\in E(G) \colon | \{x, y\} \cap A| =1 \big\}$$
and the perimeter of $A$ to be 
 \begin{equation*}
		\Per_{\nu, d}(A) := \sum_{e \in \partial_G A } \frac{ \nu(e) }{d(e)}.
	\end{equation*}

\begin{defin}[Isoperimetric dimension]\label{D:4.1.1}
For  $\delta\in [1,\infty)$, and $C_{iso}\in(0, \infty)$, $\mu$  we say that $(G, d)$ has $\nu$-\emph{isoperimetric dimension} $\delta$ with constant $C_{iso}$ if for every $A \subset V(G)$
	\begin{equation}\label{E:4.1.1.1}
		\min\{\mu(A), \mu(A^c)\}^{\frac{\delta-1}{\delta}} \leq C_{iso} \Per_{\nu, d}(A),
	\end{equation}
\end{defin}
\begin{defin}[Sobolev Norm]\label{D:4.1.2}
For $f\colon (V(G), d) \to \R$ and $p\in[1,\infty]$, we define the $(1,p)$-Sobolev norm (with respect to $\nu$ and $d$) of $f$ by 
	\begin{align*}
		\|f\|_{W^{1,p}(\nu, d)}& = \|\nabla_d f\|_{L_p(\nu)} = \E_{\nu} [|\nabla_d f|^p]^{1/p} \\
		& = \left[ \int_{E(G)} |\nabla_d f(e)|^p d\nu(e)\right]^{1/p}
	  =  \left[ \sum_{e=\{u,v\}\in E(G)} \frac{ | f(u) - f(v) |^p }{ d(u,v)^p } \nu(e)\right]^{1/p},
	\end{align*}
with the usual convention when $p=\infty$.
\end{defin}
Note that if $\nu(e)>0$ for all $e\in E(G)$, then 
$$
\|f\|_{W^{1,\infty}(\nu, d)}= \max_{e=\{u,v\}\in E(G)} \frac{|f(u)-f(v)|}{d(u,v)}=\|f\|_{\Lip}.$$

\begin{thm}[Sobolev inequality from isoperimetric inequality] \label{T:4.1.3}
 Assume that $(G, d)$ has $\nu$ - isoperimetric dimension $\delta$ with constant $C$, then for every map $f: (V(G), d) \to  \R$,
	\begin{equation}\label{E:4.1.1.2}
		\| f- \E_\mu f \|_{L_{\delta'}(\mu)} \leq 2C \| f \|_{W^{1,1}(\nu, d)},
	\end{equation}
where $\E_\mu f = \int_{V(G)} f(x)d\mu(x)$, and $\delta'$ is the H\"older conjugate exponent of $\delta$, i.e. $\frac{1}{\delta} +\frac{1}{\delta'} =1$.
\end{thm}
\begin{proof} Exercise. Hint:  \eqref{E:4.1.1.2} follows imediately from \eqref{E:4.1.1.1} if $f=1_A$.
\end{proof}

\subsection{Lipschitz- spectral profile of a graph}\label{S:4.2}
Before we define what we mean by a ``Lipschitz- spectral profile of a graph'' we want to motivate it with an example:

\begin{ex} Consider the finite abelian group $G=(\Z/n\Z)^2$, with the metric
$$d\big( (v_1,v_2), (u_1,u_2)\big)=\frac1n\max \big (|v_1-u_1|, |v_2-u_2|\big).$$

 The {\em  characters} of $G$ are 
the group homomorphisms $\chi: G\to T= \{e^{i2\pi x}: 0\le x\le 1\}$. These characters can be represented as follows:
$\chi: G\to T$ is a character if and only if for some $k,m\in \{0,1,2,\ldots n-1\}$
$$\chi=\chi_{(k,m)}: (\Z/n\Z)^2\to T,\quad (x,y)= e^{\frac{2\pi i}{n} (xk+ym)}.$$
Note the following properties of $(\chi_{(k,m)}:0\le k,m\le n)$
\begin{itemize} 
\item $(\chi_{(k,m)}:0\le k,m\le n)$ is an orthoromal basis in $L_2((\Z/n\Z)^2, \mu)$ where $\mu$ is the uniform distribution.
\item $\|\chi_{(k,m)}\|_{L_\infty(\mu)}=\|\chi_{(k,m)}\|_{L_1(\mu)}=1$,  for $0\le k,m\le n$,
\item $\|\chi_{k,m}\| \le C\max(k,m)$ and thus
$$\big|\{(k,m)\in (\Z/n\Z)^2: \|\chi_{(k,m)}\|_\Lip\le L\big\} \ge cL^2, \text{ for $L=1,2 \ldots n$}.$$
\end{itemize}
\end{ex} 

\begin{defin}[Lipschitz-spectral profile] \label{def:Lip-spec-profile}
Let  $\delta \in [1,\infty)$, and $\beta \in[1,\infty)$, and $C\ge 1$.  We say that $(G,d)$ has $(\mu, d)$-\emph{Lipschitz-spectral profile of dimension $\delta$ and bandwidth $\beta$ with constant} $C$ if there exists a collection of functions $F = \{f_i\colon V(G) \to \R\}_{i\in I}$ satisfying:
	\begin{enumerate}
		\item\label{item:1} $C^{-1} \le \inf_{i\in I} \|f_i\|_{L_1(\mu)} \le \sup_{i\in I} \|f_i\|_{L_\infty(\mu)} \le C,$ 
		\item $\{f_i\}_{i\in I}$ is an orthogonal family in $L_2(\mu)$, and
		\item\label{item:3} for every $s \in [1, \beta]$, $ |\{i \in I \colon \Lip(f_i) \leq s\}|\geq C^{-1} s^\delta$.
	\end{enumerate} 
\end{defin}

\subsection{Main result}\label{SS:4.3}

\begin{thm}\label{T:4.3}
 Let $\delta_{iso}\in [2, \infty)$, $\delta_{spec}\in [1, \infty)$ and $C\ge 1$. If $G$ has $(\nu, d)$-isoperimetric dimension $\delta_{iso}$ with constant $C$, and Lipschitz-spectral profile of dimension $\delta_{spec}$, bandwidth $\beta$, with constants $C$, then any $D$-isomorphic embedding from the Lipschitz-free space $\cF(V(G),d)$ into a finite-dimensional $L_1$-space $\ell_1^N$ satisfies 
	\begin{equation}\label{E:2.3.1}
		D \ge \frac{ 1 }{2C^5}  \left( \int_{1}^\beta s^{ \delta_{spec} - \delta_{iso} - 1}  ds \right)^{\frac{1}{\delta_{iso}}}.
	\end{equation}
\end{thm}
\begin{proof}[Sketch] Assume that $T: \cF(V(G),d)\to \ell_1^N$, is such that $\|\mu\|_\cF\le \|T(\mu)\|_1\le D \|\mu\|_\cF$ for all $\mu\in \cF(V(G),d)$.
We need to find a lower estimate for $D$. Passing to the adjoint $T^*:\ell_\infty^N\searrow \Lip_0(V(G),d)$, which is a surjection, it follows 
that
\begin{equation}\label{E:2.3.2}  B_{\Lip_0(V(G),d)}\subset T^*(\ell_\infty^N) \text{ and } \|T^*\|\le D.
\end{equation} 
We define several operators:
\begin{itemize}
\item $\iota_1: \Lip_0(V(G),d)\to W^{1,1}(V(G),d,\nu)$ identity, 1-summing operator $\pi_1(\iota_1)=1$,\\
(Note that $\nabla_d: \Lip_0(V(G),d)\to \ell_\infty(E(G))\equiv L_\infty(E(G),\nu)$, $\nabla_d: W^{(1,1)}(V(G),d,\nu_G) \to  L_1(\nu_G) $ are isometric embeddings.)
\item $\iota_2: W^{1,1}(d,\nu)\to  L_{\delta'_{iso}}(\mu)$ identity. By Sobolev inequality  $\|\iota_2\|\le 2C$.\\
\item Let $F=\{ f_j:j\in J\}$ be the  set of orthogonal function: $f_j: V(G)\to \R$, as required by the  spectral profile, and define
the {\em Fourier Transform}
$$FT: L_{2}(\mu)\to \ell_2(J), \quad g\mapsto (\E_\mu(g \cdot f_j):j\in J).$$
By the assumed orthogonality  we deduce that $\|FT\|_{ L_{2}(\mu)\to \ell_2(J)}=1$.
From the assumption on the $L_1$ - and $L_\infty$-norms of $(f_j)_{j\in J}$
$$\max_{j\in J} \big|\E_\mu(g \cdot f_j)\big\|\le C\|g\|_{L_1(\mu)} , \text{ for all $g\in L_1(\mu)$.}$$
and thus $\|FT\|_{ L_1(\mu)\to \ell_\infty(J)}\le C$.

We deduce, therefore from the Riesz-Thorin Interpolation Theorem  (Recall that $2\le  \delta_{iso}<\infty$, and thus 
$1<\delta'_{iso} \le 2$)
that $\|FT\|_{ L_{\delta'_{iso}}(\mu)\to \ell_{iso}(J)}\le C$ .
\end{itemize}
Then we consider  the product of all these operators
$$R: \ell_\infty^N\mathop{\searrow}^{\quad T^*}\Lip_0(d)\mathop{\to}^{\iota_1}W^{1,1}(d,\nu)\mathop{\to}^{\iota_2} L_{\delta'_{iso}}(\mu) \mathop{\to}^{FT} \ell_{\delta_{iso}}(J)$$
Since the summing norm  $\pi_1$ is an  {\em ideal norm}, we deduce that
$$\pi_1(R)=\pi_1(\iota_1) \|T^*\|\cdot \|\iota_2\|\cdot\|\cdot \|FT\|\le 2DC^2. $$
An important property of $1$-summing operators is that  $1$-summing operators between two Banach lattices are {\em Lattice bounded}. 
In our case, this means the following:

Let $R_j: \ell_\infty^N: \ell_\infty \to \R$ be the $j$-componenet of $R$, \ie $R(x)=\sum_{j\in J} R_j(x) e_j$, where $e_j$ is the $j$-th unit vector basis of $\ell_{\delta_{iso}}(J)$.
Then there exists a $b\in \ell^+_{\delta_{iso}}(J)$ with $\|b\|_{\delta_{iso}}\le \pi_1(R)\le 2 D C^2 $ 
so that for every $x\in \ell_\infty^N$
$$|R_j(x)\le b_j\|x\|_\infty.$$

Now, using that $B_{\Lip_0(V(G),d)}\subset T^*(B_{\ell_\infty^N})$ choose for every $j\in J$ and $x_j\in B_{\ell_\infty^N}$ so that 
$$T^*(x_j)=\frac{f_j}{\|f_j\|_{\Lip}}, \text{ for all $j\in J$.}$$
It follows that 
$$R(x_j)=\frac{R(f_j)}{\|f_j\|_\Lip}= \frac1{\|f_j\|_\Lip} e_j \text{ and } |R_j(x_j)|\le b_j\text{ for all $j\in J$.}$$
Therefore we obtain
\begin{align*} 
 2C^2 D&\ge \|b\|_{\ell_{\delta_{\text iso}}}=\Big(\sum_{j\in J} b_j^{\delta_{\text{iso}}}\Big)^{1/\delta_{\text{iso}}}
 \ge \Big(\sum_{j\in J} |R_j(x_j)|^{\delta_{\text{iso}}}\Big)^{1/\delta_{\text{iso}}}\\
 &= \Big(\sum_{j\in J} \big(\|\cF(f_j)\|^2_{L_2}/\|f_j\|_\Lip\big)^{\delta_{\text{iso}}}\Big)^{1/\delta_{\text{iso}}}
 \ge C^{-2} \Big(\sum_{j\in J} \Big(\frac1{\|f_j\|_\Lip}\Big)^{\delta_{\text{iso}}}\Big)^{1/\delta_{\text{iso}}}
 \end{align*}
 Thus
\begin{equation}\label{E:2.3.3}D\ge \frac1{2C^4} \Big(\sum_{j\in J}\Big(\frac1{\|f_j\|_\Lip}\Big)^{\delta_{\text{iso}}}\Big)^{1/\delta_{\text{iso}}}.\end{equation}
From here, we calculate the sum by applying the classical formula
$$\int_{\Omega} | h|^p d\sigma = p \int_{0}^\infty t^{p-1} \sigma( \{ |h| >t\} ) dt$$
with $\Omega = J$ and $\sigma$ the counting measure:
	\begin{align}
	\nonumber	\sum_{j \in J} \frac{ 1 }{ \Lip(f_j)^{ \delta_{iso}  } } & = \delta_{iso} \int_{0}^\infty t^{ \delta_{iso} - 1}\Big| \Big\{ j \in J \colon \frac{ 1 }{ \Lip(f_j) } > t \Big \} \Big|dt \\ 
	\nonumber	& = \delta_{iso} \int_{0}^\infty  \frac{1}{ s^{ \delta_{iso} - 1}}\Big| \Big\{ j \in J \colon \frac{ 1 }{ \Lip(f_j) } >  \frac{1}{s} \Big \} \Big| \frac{1}{s^2} ds \\
	\nonumber	& = \delta_{iso} \int_{0}^\infty  \frac{1}{ s^{ \delta_{iso} + 1}}\Big|\Big\{ j \in J \colon  \Lip(f_j) < s \Big \} \Big|ds \\
		 & {\geq}  \delta_{iso} \int_{1}^\beta  \frac{1}{ s^{ \delta_{iso} + 1} } \frac{ s^{\delta_{spec}} }{ C } ds.
	\end{align}
	(where we used (3) of the Lipschitz   spectral profile in the last inequality)
from which, together with \eqref{E:2.3.1}, we deduce our claim \eqref{E:2.3.3}.
\end{proof}
\subsection{Examples satisfying Theorem \ref{T:4.3}}
\phantom{mmmmm}

\vfill\eject
\section{Appendix}
 \subsection{Proof of Birkhoff's Theorem \ref{T:1.13}}
   \begin{thm}\label{T:A.1.13}(Birkhoff)
   Assume $n\in\N$ and  that $A=(a_{i,j})_{i,j=1}^n$ is a doubly stochastic matrix, \ie
   \begin{align*}
   &0\le a_{i,j}\le 1\text{ for all $1\le i,j\le n$,}\\
   &\sum_{j=1}^n  a_{i,j}=1\text{ for $i=1,2,\ldots, n$,}\text{ and }
   \sum_{i=1}^n  a_{i,j}=1\text{ for $j=1,2,\ldots, n$.}
   \end{align*}
   Then $A$ is a convex combination of permutation matrices, \ie matrices which have in each 
   row and each column exactly one entry whose value is  $1$ and vanishes elsewhere.
   \end{thm}
   \begin{proof}
   Clearly the set $\DS_n$ of all doubly stochastic  $n$by$n$ matrices is bounded,  convex  and compact in $\R^{n\times n}$,
   and is thus the convex hull of its extreme points.
   It is, therefore, enough to show that a matrix $A\in \DS$, which has entries that are not integers, cannot be an extreme point of $\DS$.
   
   So assume that for some $r_1$ and $s_1$ in  $\{1,2,\ldots,n \}$ we have $0<a_{r_1,s_1}<1$ and  since the row $r_1$ adds up to  $1$
   there must be an $s_2\not=s_1$ in  $\{1,2,\ldots,n \}$,   with $0<a_{r_1,s_2}<1$. Again, since the column $s_2$ adds to up to  $1$ there must
    be an $r_2\not=r_1$ in  $\{1,2,\ldots,n \}$ with $0<a_{r_2,s_2}<1$. We continue this way and
     eventually for some $k$, either  $\{r_k,s_k\}$ or $\{r_k,s_{k+1}\} $ must be among the previously chosen
     pairs $(r_j,s_j)$ or $(r_j,s_{j+1})$.
    
    Possibly by changing the starting point,  relabeling,  and exchanging rows with columns, we  obtain a cycle
    which is either of the form 
    $$(r_1,s_1), (r_1,s_2),(r_2,s_2), \ldots,(r_{k-1},s_k), (r_k,s_k) =(r_1,s_1)$$
    (implying  the cycle is of even length)
 or  of the form 
  $$(r_1,s_1), (r_1,s_2),(r_2,s_2), \ldots, (r_k,s_k),(r_k,s_{k+1})=(r_1,s_1),$$
  (implying the cycle is of odd length)\\
  so that $r_j\not=r_{j+1}$ and $s_{j}\not=s_{j+1}$ 
   (if the cycle is of the first form, we put $s_{i}=s_{i\mod(k-1)}$,  if $ I>k-1$ and 
   if it is of the second form we put  $s_{i}=s_{i\mod(k)}$,  if $i>k$) 
  and $0<a_{r_j,s_j}, a_{r_j,s_{j+1}}<1$.
  Let us assume we have chosen a cycle of minimal length. Then we claim it must be of the first form, \ie\ it must be of even length.
  
  Indeed, assume it is of the second form, then $(r_k,s_{k+1})=(r_1,s_1) $ and $(r_1,s_2)$ are in the same row and therefore  
  $$(r_2,s_2),(r_2,s_3),(r_3,s_3), \ldots,(\underbrace{r_k}_{=r_1},s_{k}), (r_1,s_2), $$
  is a shorter cycle, which is a contradiction; thus, our shortest cycle is of the first form.
  Let now $0<\vp<1$ small enough so that  for all $j\le k$
  $$\vp<\min( a_{r_j,s_j},a_{r_j,s_{j+1}}, 1-a_{r_j,s_j},1-a_{r_j,s_{j+1}})$$
  Then define
  $$b^{(1)}_{s,t}=  \begin{cases} a_{r_j,s_j}+\vp &\text{ if $(s,t)=(r_j,s_j)$ for some $j$}\\
                                                  a_{r_j,s_{j+1}}-\vp &\text{ if $(s,t)=(r_j,s_{j+1})$ for some $j$}\\
                                                    a_{s,t} &\text{ otherwise,}\end{cases}$$
  and 
  $$b^{(2)}_{s,t}=  \begin{cases} a_{r_j,s_j}-\vp &\text{ if $(s,t)=(r_j,s_j)$ for some $j$}\\
                                                  a_{r_j,s_{j+1}}+\vp &\text{ if $(s,t)=(r_j,s_{j+1})$ for some $j$}\\
                                                    a_{s,t} &\text{ otherwise.}\end{cases}$$
  It follows that $B^{(1)}=(b^{(1)})_{s,t=1}^n$ and  $B^{(2)}=(b^{(2)})_{s,t=1}^n$    are in $\DS$ and
  $$A=\frac12\big( B^{(1)}  +  B^{(2)}  \big) ,$$
  which implies that $A$ cannot be an extreme point.                                       
   \end{proof}
   
   \begin{proof}[Proof of Proposition \ref{P:1.12}] Let 
    $A=\{x_1,x_2,\ldots x_n\}$ and $B=\{y_1,y_2,\ldots, y_n\}$. 
   We note that for every $\pi\in \cP(\sigma,\tau)$ 
   the matrix\\ $M=(n \pi(x_i,y_j) : 1\le i,j\le n)$ is a doubly stochastic matrix
   (since $\sum_{x\in A} \pi(x,y) =\tau(y)=\frac1{|B|}=\frac1{|A|}=\sigma(x)=\sum_{y\in B} \pi(x,y)$).
   Thus by \eqref{E:1.3}
   \begin{align*} d_\Wa(\mu_A,\mu_B)=\frac1n \min \Big \{ \sum_{i,j=1}^n M_{i,j} d(x_i,y_j) : M\in \DS_n\Big\}
   \end{align*}
   Since the map
   $$\DS\to [0,\infty), \quad M\mapsto  \sum_{i,j=1}^n M_{i,j} d(x_i,y_j) $$
   is linear, it achieves its minimum on an extreme point, our claim follows from Theorem \ref{T:1.13} 
      \end{proof}

\subsection{Proof  of Theorem  \ref{T:3.2} on stochastic embeddings of finite metric spaces into trees}

\begin{thm}\label{T:6.1} {\rm\cite{FakcharoenpholRaoTalwar2004}} Let $M$ be a metric space with $n\in\N$ elements. Then, there is a $O(\log n)$ stochastic embedding of $M$ into the class of weighted trees. 
\end{thm}

  We need some notation:

We fix a metric space $(M,d)$. 
\begin{itemize}
 \item $B_r(x)=\{y\in M: d(x,y)\le r\}$,  for $x\in M$ and $r>0$. 
 \item For 
$A\subset M$ the {\em diameter of $A$} is 
$\diam(A)=\sup_{x,y\in A} d(x,y)$. 
\item
The set of partitions of $M$ is denoted by $\cP_M$. The elements of a partition $P$ are called {\em clusters of $P$}.
Let  $P=(A_1,A_2,\ldots A_n)\in \cP_M$. For $r\ge 0$, $P$ is called $r$-{\em bounded} if $\diam(A_j)=\max_{x,z\in A_j} d(x,z)< r$, for all $j=1,2,\ldots,n$.
For $x\in M$ and a partition $P=(A_1,A_2,\ldots A_n)$ of $M$ we denote the unique $A_j\in P$ which contains $x$ by $P_x$.
 \item A {\em stochastic decomposition } is a probability measure $\P$ on $\cP_M$, and its {\em support } is given by
 $\supp(\P)=\{ P\in\cP: \P(P)>0\}$.
 \end{itemize}
%
%

\begin{lem}\label{L:17.4} Let  $R>0$.
There is a Probability measure $\P$ on $\cP_M$ so that 
\begin{align}\label{E:17.4.1} 
&\supp(\P)\subset \big\{ P\in \cP_M: P
\text{ is $R$-bounded}\big\}\text{ and }\\
\label{E:17.4.1a} &\P\big(P\kin \cP_M: B_t(x)\ksubset P_x\big) \!\ge\! \Big(\frac{|B_{R/8} (x)|}{|B_{R} (x)|}\Big)^{8t/R}\text{ for all $0<t\le R/8$ and $x\kin M$.}
\end{align}

\end{lem}
\begin{proof} Let $M=\{x_1,x_2,\ldots,x_n\}$, let $\pi$ be a permutation on $\{1,2,\ldots, n\}$ and $r\in \big[\frac{R}4,\frac{R}2)$.
We define a partition $\tilde P(\pi, r)=\big(\tilde C_i(\pi,r)\big)_{i=1}^n$ as follows:

$\tilde C_1(\pi,r) =B_r(x_{\pi(1)})$ and assuming $\tilde C_1(\pi,r)$,  $\tilde C_2(\pi,r), \ldots ,\tilde C_{j-1}(\pi,r)$,
let $\tilde C_j(\pi,r)  =B_r(x_{\pi(j)})\setminus \bigcup_{i=1}^{j-1} B_r(x_{\pi(i)})$. Some of the $\tilde C_j(\pi,r)$, $j=1,2\ldots n$ could be empty
and we let  $P(\pi,r)=\big(C_j(\pi,r)\big)_{j=1}^m$, with $1\le m\le n$ be the non empty members  of $\big(\tilde C_j(\pi,r)\big)_{j=1}^n$,
in the  order inherited from  $\tilde P(\pi, r)$.

Let $\mu$ be the uniform distribution on $\Pi_n$ the set of all permutation on $\{1,2,\ldots, n\}$ (and thus $\mu(\pi)=\frac1{n!}$, for $\pi\in \Pi_n$) and let $\nu$ be the uniform distribution on 
$\big[\frac{R}4,\frac{R}2\big)$. Finally let $\P$ be the image distribution of $\mu\otimes \nu$ under the  map $(\pi,r)\mapsto \tilde P(\pi,r)$,
and thus
$$\P(A)=\mu\otimes \nu\big(\{(\pi,r): \tilde P(\pi,r)\in A\}\big) \text{ for }A\subset \cP_M.$$
It follows from Fubini's  Theorem  that 
\begin{align}\label{E:17.4.2}
\P(A) 
&= \int_{R/4}^{R/2}\int_{\Pi_n} 1_{\{(\pi',r'): \tilde P(\pi',r')\in A\}}(\pi,r) \,d\mu(\pi) \,d\nu(r)\\
&=\int_{R/4}^{R/2}    \mu(\{ \pi\kin \Pi_n:\tilde P(\pi,r)\kin A\})\,d\nu(r)\notag\\
              &=\frac4R\int_{R/4}^{R/2}    \mu(\{ \pi\kin \Pi_n:\tilde P(\pi,r)\kin A\})\,d(r).\notag
              \end{align} 

\noindent{\bf Claim.} For $\frac{R}4\le r<\frac{R}2$, $0<t\le r$ and $x\in M$ it follows that
\begin{equation}\label{E:17.4.3} \mu\big( \{ \pi\in  \Pi_n: B_t(x) \subset \tilde P_x(\pi,r) \big)\ge \frac{|B_{r-t} (x)|}{|B_{r+t} (x)|}.
\end{equation}
In order to prove the claim we order $B_{r+t}(x) $ into $(y_i)_{i=1}^a$  and $(y_i)_{i=a+1}^b$
so that 
$$0=d(x,y_1)\le d(x,y_2)\le  \ldots \le d(x,y_a)\le r-t< d(x,y_{a+1})\le  \ldots \le  d(x,y_b)\le r+t.$$
Thus $|B_{r-t} (x)|=a$,  $|B_{r+t} (x)|=b$, and 
$y_i\in  B_{r-t} (x)$ if $1\le i\le a$, and 
$y_i\in B_{r+t} (x)\setminus B_{r-t} (x)$ if $a<i\le b$. 
Define
$$E_i:= \left\{\pi\in \Pi_n: \begin{matrix} 
                                                               y_{\pi(s)}\not\in                   \{ y_1,y_2,\ldots y_b\} \text{ for $s=1,2\ldots,i-1$}\\
                                                               y_{\pi(i)}\in \{ y_1,y_2,\ldots y_a\} 
                                                                    \end{matrix}\right\}.$$
 In other words $E_i$ is the event that $i$ is the smallest $j\le n$ for which $y_{\pi(j)} \in B_{r+t}(x)$ intersected with the event that $y_{\pi(i)}\in   B_{r-t}(x  )$.
                                                            
Since for   $\pi\in E_i $ and  $s<i$ we have $d(y_{\pi(s)}, x)>r+t$  and  $d(y_{\pi(i)},x)<r-t$ it follows that $B_t(x) \subset \tilde P_x(\pi,r)$.
For $i=1,2,\ldots n$ let
$$A_i=\left\{\pi\in \Pi_n: \begin{matrix} 
                                                               y_{\pi(s)}\not\in                   \{ y_1,y_2,\ldots y_b\} \text{ for $s=1,2\ldots,i-1$}\\ y_{\pi(i)}\in \{ y_1,y_2,\ldots y_b\} 
                                                               \end{matrix}\right\}.$$
                                                               
Then the  sets $(A_i)_{i=1}^n $ are a partition of $\Pi_n$, and $E_i\subset A_i$, and moreover $\mu(E_i|A_i)=\frac{a}{b}$
(since  $\mu(E_i|A_i)$ is the probability  that $\pi(i) \le a$ assuming that  $\pi(i) \le b$).
Thus  
$$\mu(\{ \pi\in \Pi_n: B_t(x)\ksubset \tilde P_x(\pi,r)\}) \ge\mu \Big(\bigcup_{i=1}^n E_i\Big)=\sum_{i=1}^n \mu(A_i) \mu(E_i|A_i) =\frac{a}{b}= \frac{|B_{r-t} (x)|}{|B_{r+t} (x)|},$$
which proves our claim.

To finish the proof of the Lemma we deduce from \eqref{E:17.4.2} and \eqref{E:17.4.3} for $0<t\le R/8$ that
\begin{align*}
\P(\{\pi\in \Pi_n: B_t(x)\subset P_x(\pi,r)\}) &= \frac4R\int_{R/4}^{R/2} \mu\big(\{\pi\in \Pi_n: B_t(x) \subset P_x(\pi,r)\}\big) \, dr\\
&\ge\frac4R\int_{R/4}^{R/2}{\frac{|B_{r-t} (x)|}{|B_{r+t} (x)|}}\, dr\\
&=\frac4R\int_{R/4}^{R/2} e^{ h(r-t)-h(r+t)}\, dr\\
&\text{ where $h(s) =\log\big( |B_s(x)|\big)$}\\
&\ge\exp\Big(  \frac4R\int_{R/4}^{R/2} h(r-t)-h(r+t)\, dr\Big)\\ &\text{ (By Jensen's inequality)}\\
&=\exp\Big( \frac4R\int_{\frac{R}4-t}^{\frac{R}2-t}  h(s)\,ds- \frac4R\int_{\frac{R}4+t}^{\frac{R}2+t} h(s)\,ds\Big)\\
&=\exp\Big(\frac4R\int_{\frac{R}4-t}^{\frac{R}4+t} h(s)\,ds-\frac4R\int_{\frac{R}2-t}^{\frac{R}2+t}  h(s)\, ds\Big)\\
&\ge \exp\Big(\frac{8t}R\Big( h\Big(\frac{R}{4}-t\Big) -h\Big(\frac{R}2+t\Big)\Big)\Big)\\
&\ge\exp\Big(\frac{8t}R\Big(h\Big(\frac{R}{8}\Big)-h(R)\Big)\Big)
= \Bigg(\frac{|B_{R/8}(x)|}{|B_R(x)|}\Bigg)^{\frac{8t}R}.
\end{align*}
\end{proof}
%

\begin{cor}\label{C:17.5a} For  the probability measure  $\P$ defined in Lemma \ref{L:17.4} it follows that 
$$\P\big( \{ P\in \cP_M :B_t(x)\not\subset P_x\}\big)\le  \frac8{R} t\log\Big(\frac{|B_R(x)|}{|B_{R/8}(x)|}\Big),\text{ for all $t\in(0,R/8)$.} $$
\end{cor}
\begin{proof}  We deduce from Lemma \ref{L:17.4} that 
\begin{align*}
\P\big( \{ P\in \cP_M :  B_t(x)\not\subset P_x\}\big)&=  1- \P\big( \{ P\in \cP_M :  B_t(x)\subset P_x\}\big)\\
                                                                &\le 1- \Big(\frac{|B_{R/8}(x)|}{|B_R(x)|}\Big)^{\frac{8t}R}\\
                                                &\le \log\Bigg( \Big(\frac{|B_R(x)|}{|B_{R/8}(x)|}\Big)^{\frac{8t}R}\Bigg)= \frac8{R} t\log\Big(\frac{|B_R(x)|}{|B_{R/8}(x)|}\Big),
\end{align*} 
where the last inequality follows from the fact that for $z\ge 1$ 
$$1-\frac1z\le \log(z).$$
\end{proof}
\begin{proof}[Proof of Theorem \ref{T:6.1}] After rescaling we can  assume that $d(x,y)\ge1$ for all $x\not =y$, in $M$, and  choose $k\in \N$ so that $\diam(M)\in [2^{k-1}, 2^k)$. 
We will first introduce some notation:

For two partitions $P$ and $Q$  of the same set $S$, we say $P$ {\em subdivides }$Q$ and write $P\succeq Q$
if every cluster $A\in Q$  is a union of clusters in $P$. If $A\subset S$ then the {\em restriction of $P$ onto $A$},
is the 
$$P|_A=\{ B\cap A: B\in P\}\setminus \{\emptyset\}.$$

For a metric space $M$ and any $R>0$, we denote the probability measure on the set $\cP$ on $R$-bounded Partitions 
constructed in Lemma \ref{L:17.4} by $\P^{(M,R)}$.

Our probability will be defined on the set 
$$\Omega=\big\{ (P^{(j)})_{j=0}^k\subset \cP_M: P^{(j)} \text{ is $2^{k-j}$ bounded and } P^{(j)}\succeq P^{(j-1)}, \text{ for }j=1,2,\ldots,k\big\}.$$
We let $\P$  be the probability measure on the subsets of $\Omega$, uniquely defined by the following properties:
\begin{align}
\label{E:17.1}&\P\big(\{ \bar P=(P^{(j)})_{j=0}^k\in \Omega : P^{(0)}=M\big\})=1\\
\intertext{ for $i=1,2, \ldots n$, $B\subset M$, $\diam(B)\le  2^{k-(i-1)}$ and $\cA\subset \cP_B$ we have }
\label{E:17.2}&\P\big( P^{(i)}|_B \in \cA \, \big|  B\text{ is cluster of } P^{(i-1)} \big)= \P^{(B,2^{k-i})}(\cA).
\end{align}  
Condition \eqref{E:17.2} means the following: under the condition that $P^{(i-1)}$ is given and contains a cluster $B\subset M$, the distribution of $P^{(i)}$ restricted to $B$ 
is $\P^{(B,2^{k-i})}$. So we can think of $(P^{(j)})_{j=0}^k$ as a stochastic process  with values in $\cP_M$,
whose distribution  is determined by  transition  probabilities:
$P^{(0)}=M$ (the trivial partition and given $P^{(i-1)}$, we consider each cluster  $B$ of $P^{(i-1)}$, and randomly divide $B$ according to the distribution $\P^{(B,2^{k-i})}$.
Since $d(x,y)\ge1$ for $x\not=y$, and since $\diam(M)<2^k$, it follows that $P^{(k)}$ is the {\em finest partition}, \ie\  $P^{(k)}=\big\{ \{x\}: x\in M\big\}$.

By induction on $k\in\N$ it can easily be seen that the probability $\P$  exists and is uniquely defined by the above property.

For each $\bar P= (P^{(j)})_{j=0}^k \in \Omega$ we define the following weighted tree $T=T(\bar P)=(V(\bar P), E(\bar P) , W(\bar P))$,
where 
\begin{align*}
&V=\bigcup_{j=0} ^k V_j \text{ with } V_j=\{(j, B): B \text { is cluster of $P^{(j)}$}\}, \text{ for $j=0,1,\ldots,n$}\\
&E=\bigcup_{j=1} ^k E_j \text{ with } E_j=\big\{ \{(j,A),(j-1,B)\}: A\in V_j, B\in V_{j-1},\text{ and } A\subset B\big\} \\
&\qquad\qquad \text{ for $j=1,\ldots,n$,}\\
&W: E\to \R, \quad e\mapsto 2^{k-j} \text{ if $e\in E_j$}.
\end{align*}
One might ask why we did not simply define the vertices of $T$ as the set of all the clusters of the $P^{(j)}$.
The problem is that the partitions $P^{(j)} $ and   $P^{(j-1)} $ could share the same clusters, and we need to distinguish them.

We note that $T$ is a tree, and  that then $(k,\{x\})$, $x\in M$ are the leaves  of $T$.

For $\bar P$ we define 
$$f_{\bar P}: M\to  T_{\P} , \quad x\mapsto  (k,\{x\}).$$
We claim that for some universal constant $D$, $(f_{\bar P} )_{\bar P\in \Omega} $ with the coefficients $(\P(\bar B))_{\bar B\in \Omega} $
is a $D\log(n)$ stochastic embedding.

Let us first show the lower bound. Assume that  $\bar P=(P^{(j)})_{j=0}^k\in \Omega$ and $x\not= y$ are in $M$. Then there is an $i\in\{1,2,\ldots,k\} $ so that 
$x,y$ are in the same cluster of $P^{(i-1)} $ but in two different clusters of  $P^{(i)} $.
This implies that $d(x,y)\le 2^{k-(i-1)}=2^{k+1-i} $ and and for the tree metric of $T_{\bar P}$, 
it follows that 
$$d_{T_{\bar P}}(f_{\bar P}(x),    f_{\bar P}(y) )=2\sum_{j=i}^k 2^{k-j} =2\sum_{j=0}^{k-i} 2^{j} =  2 (2^{k+1-i}-1)\ge d(x,y).$$

To get the upper estimate, first note that  from Corollary  \ref{C:17.5a}
it follows  for $i=1,2,\ldots k$, and $x,y\in M$, if $d(x,y) <2^{k-i-3}$
\begin{align}\label{E:17.5}
\P( P^{(i-1)}_x\keq P^{(i-1)}_y &\text{ and } P^{(i)}_x\not= P^{(i)}_y)\\
&= \P(   P^{(i)}_x\not= P^{(i)}_y\big| P^{(i-1)}_x\keq P^{(i-1)}_y )  \P( P^{(i-1)}_x\keq P^{(i-1)}_y)\notag\\
&\le  \P(   P^{(i)}_x\not= P^{(i)}_y\big| P^{(i-1)}_x= P^{(i-1)}_y )\notag \\
&=\P^{(P^{i-1}_x,2^{k-i})}(P\in \cP_M : y\not\in P_x)\notag\\
&\text{(Apply \eqref{E:17.2} to $B=P^{i-1}_x=P^{i-1}_y$)}\notag\\
&\le \P^{(P^{i-1}_x,2^{k-i})}(P\in \cP_M : B_{d(x,y)}\not\subset P_x)\notag\\
&\le \frac{8}{2^{k-i} }\log\Big( \frac{| B_{2^{k-i}}(x)|}{|B_{2^{k-i-3}}(x)|}\Big) d(x,y). \notag
\end{align}
It is also true that if  $d(x,y) >2^{k-i+1}$ then 
$$\P( P^{(i-1)}_x= P^{(i-1)}_y \text{ and } P^{(i)}_x\not= P^{(i)}_y)\le \P( P^{(i-1)}_x= P^{(i-1)}_y)= 0.$$

Secondly, note that for 
$\bar P\in \{ (P^{(j)})_{j=0}^k\in \Omega:  \{P^{(i-1)}_x= P^{(i-1)}_y \text{ and } P^{(i)}_x\not= P^{(i)}_y\}$
it follows that 
$$d_{T(\bar P)}(f(x),f(y))=2\sum_{j=0}^{k-i} 2^{j} =2( 2^{k-i+1}-1)< 2^{k+2-i}.$$
Let $s\le k$ so that $2^{s-1}<d(x,y) \le 2^s$. We compute 
\begin{align*}
\E_\P(d_{T(\bar P)}(f(x),f(y)))&\le \sum_{i=1}^k 2^{k+2-i}\P( P^{(i-1)}_x= P^{(i-1)}_y \text{ and } P^{(i)}_x\not= P^{(i)}_y)\\
&= \underbrace{\sum_{i=1}^{k-s-3} 2^{k+2-i}\P( P^{(i-1)}_x= P^{(i-1)}_y \text{ and } P^{(i)}_x\not= P^{(i)}_y)}_{=\Sigma_1}\\
&\quad+\underbrace{\sum_{i=k-s-2}^{k-s} 2^{k+2-i}\P( P^{(i-1)}_x= P^{(i-1)}_y \text{ and } P^{(i)}_x\not= P^{(i)}_y)}_{=\Sigma_2}\\
&\quad+\underbrace{\sum_{i=k-s+1}^k 2^{k+2-i}\P( P^{(i-1)}_x= P^{(i-1)}_y \text{ and } P^{(i)}_x\not= P^{(i)}_y)}_{=\Sigma_3}.\\
\end{align*}
Since $d(x,y)>2^{s-1} $ it follows that $\Sigma_3=0$.
Secondly 
$$\Sigma_2\le 3\cdot  2^{k+2-(k-s-2)}= 3\cdot 2^{s+4}\le 3 \cdot 2^5 d(x,y)$$
To estimate $\Sigma_1 $ we are able to use \eqref{E:17.5}  since $d(x,y)\le 2^s=\frac18  2^{k-(k-s-3)} $
\begin{align*}
\Sigma_1&\le  \sum_{i=1}^k 2^{k+2-i}\frac{8}{2^{k-i} }\log\Big( \frac{| B_{2^{k-i}}(x)|}{|B_{2^{k-i-3}}(x)|}\Big) d(x,y)\\
&={32}\log\Big(\prod_{i=1}^k  \frac{|B_{2^{k-i}}(x)|}{|B_{2^{k-i-3}}(x)|} \Big) d(x,y)\\
&\le 32 \log(n^3)d(x,y)=96 \log(n) d(x,y)
\end{align*}
And thus 
$$\E_\P(d_{T(\bar P)}(f(x),f(y)))\le 96 \log(n) d(x,y)+ 3\cdot 2^5 d(x,y).$$
\end{proof}

\subsection{Proof of  Theorem \ref{T:3.3}  }

Our goal is to show that for some universal constant $c$, every metric space with $n$ elements can be bijectively $c\log(n)$-stochastically embedded into trees. This will follow from Theorem \ref{T:3.2}  and the following result by Gupta.
\begin{thm}\label{T:13.4} {\rm \cite[Theorem 1.1]{Gupta2001}}  Let $T=(V,E,W)$ be a weighted finite tree and $V'\subset V$. Then there is $E\subset [V']^2$ and $W': E(\G')\to [0,\infty)$
so that $T'=(V',E(\G'),W')$ is a tree 
\begin{equation}\label{E:13.4.1} \frac14\le \frac{d_{T'}(x,y)}{d_{T}(x,y)}\le 2,\text{ for $x,y\in V'$}.\end{equation}
\end{thm}

We first show the claim of Theorem \ref{T:13.4} in the special case that $V'$ consists of leaves of $T$. Recall that in a tree $T=(V,E)$
$$\Leaf(T)=\{ v\in V: \deg_T(v)=1\}.$$

\begin{lem}\label{L:13.5} Let $T=(V,E,W)$ be  a weighted finite tree and $V'\subset \Leaf(T)$, and let $d_T$ be the geodesic metric generated by $W: E\to (0,\infty)$. Then there is $E(\G')\subset [V']^2$ and $W': E'\to (0,\infty)$
so that $T'=(V',E',W')$ is a tree and 
$$\frac14\le \frac{d_{T'}(x,y)}{d_{T}(x,y)}\le 2,\text{ for $x,y\in V'$}.$$
\end{lem} 
\begin{proof} 
We start with a general weighted tree $T=(V,E,W)$ and a subset  $V'$ of $V$.
By eliminating successively every leaf of $T$ which is not in $V'$, we can assume that $V'$ are all the leaves. 
Unless $V'=V$ (in which case we are done), there must be 
an element $x_0\in V\setminus V'$,
 and the degree of this element must be at least $2$.
Denote the partial order defined by letting $x_0$ be the root of $T$  by $\succeq$. For $x,y\in V$ we denote 
by $x\wedge y$ the {\em minimum of $x$ and $y$} meaning the maximal vertex $z$  with respect to $\succeq$ 
for which $x\succeq z$ and $y\succeq z$.

Let $v_0\in V'$ for which $r_0:=d_{T}(x_0,v_0)$ is minimal. 
Let $\tilde E\subset E$ be the set of  edges $e=\{a,b\}$ in $E$ for 
which $d_T(x_0,a)< r_0/2\le d_T(x_0,b)$.
Order $\tilde E$ into $\{e_1,e_2,\ldots,e_n\}$, $e_i=\{a_i,b_i\}$ with $d_T(x_0,a_i)< r_0/2\le d_T(x_0,b_i)$.
One of the edges in $\tilde E$ must be contained in $[v_0, x_0]$ and assume that $\tilde E$ was ordered so that $e_1\subset[v_0,x_0]$.

We now define new trees $T_1,T_2,\ldots, T_n$, which are, up to possibly one additional element subtrees of $T$.

If $r_0/2=  d_T(x_0,b_j)$  we let $T_j$ be the subtree   $T_j=(V_j,E_j,W_j)$ with $$V_j=\{x\in V: x\succeq b_j\},\quad
E_j=E\cap[V_j]^2 \text{ and }W_j=W|_{E_j}.$$ In that case we put $x_j=b_j$.

 If $r_0/2<  d_T(x_0,b_j)$ then $T_j=(V_j,E_j,W_j)$, with $V_j= \{x\in V: x\succeq b_j\}\cup \{x_j\}$ where
 $x_j$ is an element not in $V$, and distinct from all the other $x_i$
  and we let $$E_j=E\cap[V_j]^2\cup \big\{\{ x_j, b_j\}\big\} \text{ and }
  W_j(e)=\begin{cases}  W(e)  &\text{if $e\in E\cap[V_j]^2$,}\\
                                                        d(x_0,b_j)-r_0/2 &\text{if $e=\{ b_j,x_j\}$.} \end{cases} $$
     
     We also define the following tree $\bar T=(\bar V, \bar E, \bar W)$  which contains $T$, and $T_1$, $T_2,\ldots T_n$ isometrically:
     \begin{align*}\bar V &=V\cup\{ x_j: j=1,2\ldots, n\}=V\dot\cup\{ x_j: j=1,2\ldots, n, d_T(x_0, b_j)>r_0/2\}\\
 \bar E&=(E\setminus \tilde E)\cup\big\{\{a_j,x_j\}:j=1,2\ldots n\} \big\}\cup \{\{b_j,x_j\}:j=1,2\ldots n, b_j\not= x_j\} \\
     \bar W&:\bar E\to (0,\infty_,\quad e\mapsto \begin{cases}  W(e) &\text{ if $e\in E\setminus E'$},\\
                                                                                                   r_0/2- d_T(a_j, x_0)&\text{ if $e= \{ a_j, x_j\} $},\\
                                                                                                    d_T(b_j,x_0) -r_0/2 &\text{ if $e= \{ b_j, x_j\} $ and $b_j\not= x_j$.}\end{cases}
                                                                                                    \end{align*}
                                                        
We note that the inclusion $(V,d_{T} )\subset  (\bar V,d_{\bar T})$, $(V_i,d_{T_i} )\subset  (\bar V,d_{\bar T})$ are isometric embeddings.

Let us make some observations:
\begin{enumerate}
\item Since for $j=1,\ldots n$ the vertex $a_j$ lies on the path $[x_0,b_j]$ connecting $x_0$ and $b_j$, it follows 
if $b_i=b_j$ then also $a_i=a_j$ and thus $i=j$. Thus, all the $b_j$, $j=1,\ldots n$,  are pairwise distinct, and the  $V_j$ are pairwise disjoint.

\item We claim that $n$ is at least $2$. Indeed, let $z_1$ and $z_2$ be two direct successors of $x_0$ (the degree of $x_0$ is at least $2$)
and let $w_1$ be a leaf in $\{ x\in V: x\succeq z_1\}$ and $w_2$ be a leaf in $\{ x\in V: x\succeq z_2\}$,
and let for $s=1,2$  $e_{i_s}=\{a_{i_s},b_{i_s}\}$ be the edge in $[x_0, w_s]$ for which $d_T(x_0,a_{i_s})< r_0/2\le d_T(x_0,b_{i_s})$
(such edges exist because of the minimality of $r_0$).
Since the path from $w_1$ to $w_2$ must go through $x_0$ it follows that $b_{i_1}\not=b_{i_2}$ (it could be possible that $a_{i_1}=a_{i_2}=x_0$!) 

\item Put  $V'_j=V'\cap V_j$ for $j=1, \ldots n$. Since for all $w\in V'$ $d(x_0,w)\ge r_0$, it follows 
that $V'=\bigcup_{j=1}^n V'_j$, and that the $V_j'$ are the leafs of $T_j$.

\item For $i\not= j$ and $v\in V_i'$ and $w\in V'_j$, we observe that $a_i$ and $a_j$ have to lie on the path connecting   $v$ with $w$ and thus 
\begin{align*}
 d_T(v,w)\ge d_T(v, a_i)+d_T(w,a_j)
\end{align*}

\end{enumerate}

For $j=1,2,\ldots n$  put $T_j=(V_j,E_j, W_j)$ with $E_j=\bar E\cap [V_j]^2$ and $W_j=\bar W|_{E_j}$,  and $V_j'$ is of strictly less cardinality 
that $V'$. Moreover for $x,y\in V_j$ we have $d_{\bar T}(x,y)=d_{T_j}(x,y)$.

For $j=1,\ldots n$, put $r_j=\min_{v\in V_j'} d_{T_j}(x_j, v)$, and choose $v_j\in V_j$ so that $d_{T_j}(v_j,x_j)=r_j$.
Note that $r_1=r_0/2= d_{T_j}(v_0, x_1)$ and thus we can assume that $v_1=v_0$.

Then we can continue to decompose  $T_1$, $T_2,\ldots T_n$ until we arrive at trees that consist of one single element of $V'$ by applying inductively the following claim.

We claim the following:  

\noindent{\bf Claim:} Assume  that for $j=1,2\ldots n$, $T_j$ satisfies the  following condition:
There is a tree $T'_j=(V'_j,E'_j,W'_j)$ on $V'_j$ so that 
\begin{align}
\label{E:13.6.1}d_{T'_j}(x,v_j)\le 2d_{T_j} (x,x_j)- r_j ,\text{ for $x\in V_j'$ }\\
\label{E:13.6.2}\frac14\le\frac{d_{T'_j}(x,y)}{d_{T_j}(x,y)} \le 2,\text{ for $x,y\in V_j'$ }
\end{align}  

Then construct $T'=(V',E',W')$ by  {\em glueing } $T'_1$, $T'_2,\ldots T'_n$ together, connecting $v_1=v_0$ to the other $v_j$. More precisely, we put 
$$E'=\bigcup_{j=1}^n E'_j \cup \big\{\{v_1,v_j\}: 2\le j\le n\big\}$$
and use the weight
$$W'(e)=W'(e) \text{ if }e\in \bigcup_{j=1}^n E'_j\text{ and }W(\{v_1,v_j\})=d(x_0,v_j))=r_j.$$
We claim that it  follows that  
\begin{align}
\label{E:13.6.3}&d_{T'}(x,v_0)\le 2d_{T} (x,x_0)- r_0 ,\text{ for $x\in V'$, }\\
\label{E:13.6.4}&\frac14\le\frac{d_{T'}(x,y)}{d_{T}(x,y)} \le 2,\text{ for $x,y\in V'$.}
\end{align}  
Since \eqref{E:13.6.3} and \eqref{E:13.6.4} are clearly satisfied if $V'$ is a singleton, the  Theorem follows by induction from the claim.

We first verify the first inequality in \eqref{E:13.6.4}: For $x,y\in V'$, either $x,y$  lie  both in $V_j'$ for some $j$, we deduce our claim from the induction hypothesis
or  $x\in V_i $ and $y\in V_j$, $1\le i, \le n$, $i\not =j$, and so without loss of generality $j\not=1$
Then 
\begin{align*}d_{T}(x,y)&\le  d_{T} (x,v_i)+ d_{T} (v_i,x_0)+d_{T} (v_j,x_0)+d_{T}(v_j,w)\\
&\le 4d_{T'} (x,v_i)+4d_{T'}(v_j,y)+ d_{T_i} (v_i,x_i)+d_{T_j} (v_j,x_j)+ d_{\bar T} (x_i,x_0)+d_{\bar T} (x_j,x_0)\\
&\text{(By the induction Hypothesis)}\\
&= 4d_{T'} (x,v_i)+4d_{T'}(v_j,y)+ r_i+r_j +r_0\\
&= 4(d_{T'} (x,v_i)+d_{T'}(v_j,y))+ r_i+r_j+ 2r_1\\
&= 4(d_{T'} (x,v_i)+d_{T'}(v_j,y))+\begin{cases} 3r_1 +r_j &\text{if $i=1$}\\ 
                                                                                  2r_i+ 2r_j&\text{if $i\not =1$}\end{cases} \\
&\le 4d_{T'} (x,y).
\end{align*} 
Secondly, we verify the second inequality in \eqref{E:13.6.4}. Let $x,y\in V'$. If $x,y$ lie both in $V_j'$ for some $j$; we deduce our claim again from the induction hypothesis. So assume that $x\in V_i $ and $y\in V_j$, $1\le i, \le n$, $i\not =j$, Also assume that $j\not=1$.

We deduce that 
\begin{align*}d_{T'}(x,y)&= d_{T'_i} (x,v_i)+d_{T'}(v_i,v_j)+d_{T_j'}(y,v_j)\\
                   &\le d_{T'_i} (x,v_i)+d_{T_j'}(y,v_j)+r_i+r_j\\
                   &\le 2d_{T_i}(x, x_i)-r_i +2d_{T_j}(y, x_j)-r_j +r_i+r_j\\
                   &= 2d_{T_i}(x, x_i) +2d_{T_j}(y, x_j)\\
                   &\le 2d_{T}(x,a_i)+2d_{T}(y,a_j) \le 2d_T(x,y).\end{align*}

We finally verify \eqref{E:13.6.3}. Let $x\in V'$ and thus $x\in V'_j$ for some $j=1,2\ldots n$.

\noindent{\bf Case 1:} $j=1$. Thus $x$ and $v_1=v_0$ are in $T_1$ and  it follows that
\begin{align*}
d_{T'} (x, v_1)&=d_{T'_1}(x,v_1)  \text{ by definition of $T'$ $[x,v_1]\subset V_1'$}\\
 &\le 2 d_{\bar T}(x,x_1)   -r_1\text{ (by \eqref{E:13.6.1} )}\\
 &\le 2 (d_T(x , x_0) -r_0/2) -r_1/2    \\
 &  =   2d_T(x, x_0) -r_0. 
\end{align*}
\noindent{\bf Case 2:} $j>1$. Then
\begin{align*}
d_{T'}(x,v_1)&= d_{T'_j} (x,v_j) + d_{T'}(v_j,v_1)\\
&\le 2 d_{T_j}(x,x_j) -r_j +r_j\text{ (by \eqref{E:13.6.1} )}\\
&=2 d_{T_j}(x,x_j)\\
&=2(d_{T}(x,x_0)-r_0/2)=2d_{T}(x,x_0)-r_0.
\end{align*}
\end{proof} 
\begin{proof}[Proof of Theorem \ref{T:13.4}]
 As in the proof of Lemma  \ref{L:13.5}, we can assume that $\Leaf(T)\subset V'$.
 We will prove the result by the induction of the cardinality of the set 
$V'\setminus \Leaf(T)$ Lemma  \ref{L:13.5} handles the case that $V'\subset \Leaf(T)$.

Assume $a\in V'\setminus \Leaf(T)$. We choose $a$ as the root and denote the partial order of $T$, if by $\succeq$. 
Let $n$ be the degree of $a$ and $a_1,a_2, \ldots a_n$ be the neighbors of $a$. We define trees $T_j=(V_j,E_j,W_j)$  and $V_j'\subset V_j$,  for $j=1,2,\ldots n$
by 
$$V_j=\{ x\in V: x\succeq a_j\}\cup\{a\}, E_j=E\cap [V_j]^2,  W_j=W|_{E_j}\text{ and } V_j'=V'\cap V_j$$
Then the cardinality of $V_j'\setminus \Leaf(T_j)$ is strictly smaller than the cardinality 
of $V'\setminus \Leaf(T)$, and we can apply the inductive hypothesis to obtain for each $j=1,2\ldots n$ a tree $T'_j=(V_j', E_j', W'_j)$ satisfying 
\eqref{E:13.4.1}.

Thus since the sets $V_j'\setminus\{a\}$, are pair wise disjoint and $a\in V_j$, for $j=1,2\ldots,n$ 
$T'=(V',E', W')$, with $E'=\cup_{j=1}^n E'_1$, $W':E'\to  (0,\infty), e\mapsto W_j'(e)$, if $e\in E'_j$ is also a tree.
Since every path from an element in $V_i$ to and element in $V_j$ has to contain $a$ it follows that $T'=(V',E',W')$ satisfies 
\eqref{E:13.4.1}.
\end{proof}

Using Theorem \ref{T:6.1} and Corollary \ref{T:13.4} we obtain the following results
\begin{cor}\label{C:13.7a}
Assume that $(M,d)$ embeds $D$-stochastically into trees, then $(M,d)$ embeds bijectively  $8D$-stochastically int trees.

Thus there are $n$, so that for $i=1,2,\ldots$ there is a set $E_i\subset[M]^2$, a metric $d_i$, and numbers $p_i\in (0,1]$, 
so that $\sum_{i=1}^n p_i=1$, $T_i=(M,E_i)$ is a tree, and $d_i$ a geodesic distance on $M$ with respect to $T_i$ and so that 
for all $x,y\in M$
\begin{align}\label{E:13.7a.1} 
&d(x,y)\le d_i(x,y), \text{ for $i=1,2,\ldots, n$,}\\
&\sum_{j=1}^n p_id_i(x,y)\le Dd(x,y).\label{E:13.7a.2}
\end{align}
Moreover we can assume that for $i=1,2,\ldots n$ actually $d_i$ is the geodesic metric generated by the 
weight function
$$w_i: E_i\to [0,\infty),\quad   \text{ with } w_i(e) =d(u,v)\text{ if } e=\{u,v\}\in E_i.$$
\end{cor}
\begin{proof} The existence of $n$  sets $E_i\subset[M]^2$, a metrics $d_i$, and numbers $p_i\in (0,1]$, for $i=1,2,\ldots,n$ follows from
Theorem \ref{T:6.1} and Corollary \ref{T:13.4}. To see the moreover part,  let us denote  for $i=1,2,\ldots, n$ the geodesic distance on $M$
generated by the above defined weight function $w_i$ by $\tilde d_i$. From the triangle inequality, it follows that $\tilde d_i(x,y)\ge d(x,y)$.
 We note that from  \eqref{E:13.7a.1} it follows that 
for every $i=1,2,\ldots,n$, any $e=\{u,v\} \in E_i$, we have that $d(u,v)\le d_i(u,v)$, and thus $d(x,y)\le \tilde d_i(x,y)\le d_i (x,y)$, for any $x,y\in M$,
which implies \eqref{E:13.7a.2}.
\end{proof}

\end{document}